\documentclass[12pt]{amsart}

\usepackage{graphicx}
\usepackage{amsmath,amsthm,amscd,amssymb, color}

\definecolor{pink}{rgb}{1,.2,.6}

\definecolor{orange}{rgb}{0.7,0.3,0}

\definecolor{blue}{rgb}{.2,.6,.75}

\definecolor{green}{rgb}{.4,.7,.4}

\newcommand{\ccom}[1]{{\color{pink}{Chantal: #1}} }

\usepackage{amssymb} 
\newtheorem{lemma}{Lemma}
\newtheorem{theorem}[lemma]{Theorem}
\newtheorem{corollary}[lemma]{Corollary}
\newtheorem{identity}[lemma]{Identity}
\newtheorem{conjecture}[lemma]{Conjecture}
\numberwithin{equation}{section} \numberwithin{lemma}{section}

\newcommand{\kommentar}[1]{}

\newcommand{\RR}{\mathbb{R}}
\newcommand{\QQ}{\mathbb{Q}}
\newcommand{\CC}{\mathbb{C}}

\newcommand{\ZZ}{\mathbb{Z}}

\newcommand{\FF}{\mathbb{F}}
\newcommand{\C}{{\mathcal{C}}}
\newcommand{\F}{{\mathcal{F}}}

\newcommand{\modd}{\text{~mod~}}
\newcommand{\RRe}{{\rm Re}}
\newcommand{\IIm}{{\rm Im}}

\newcommand*{\reff}[1]{(\ref{#1})}

\newcommand{\FX}{{\mathcal{F}}(X)}

\newcommand{\Eab}{E_{a,b}}
\newcommand{\lab}{\lambda_{a,b}}
\newcommand{\al}{\alpha}
\newcommand{\ga}{\gamma}

\newcommand{\denominator}{p^{m_1(\frac{1}{2} + \al) + m_2(\frac{1}{2} + \ga)}}
\newcommand{\denominatoralpha}{p^{m_1(\frac{1}{2} + \al)}}

\newcommand{\summoneiseven}{\sum_{\substack{m_1\geq 0 \\m_1 \;\rm even}}}
\newcommand{\summoneisodd}{\sum_{\substack{m_1\geq 1 \\m_1 \;\rm odd}}}
\newcommand{\summoneisevenpositiv}{\sum_{\substack{m_1 > 0 \\m_1 \;\rm even}}}
\newcommand{\summtenisevenpositiv}{\sum_{\substack{m_1 \geq 10 \\m_1 \;\rm even}}}
\newcommand{\summoneisoddpositiv}{\sum_{\substack{m_1 > 0 \\m_1 \;\rm odd}}}

\newcommand{\summnineisoddpositiv}{\sum_{\substack{m_1 \geq 9 \\m_1 \;\rm odd}}}
\newcommand{\tQprt}{\widetilde{Q}^*_{r,t}}
\newcommand{\tQpzero}{\widetilde{Q}^*(p^{m_1}, p^0)}
\newcommand{\tQpone}{\widetilde{Q}^*(p^{m_1}, p^1)}
\newcommand{\tQptwo}{\widetilde{Q}^*(p^{m_1}, p^2)}

\newcommand{\tQ}{\widetilde{Q}^*}

\newcommand{\Ft}{\mathcal{F}}
\newcommand{\Qt}{\widetilde{Q}^*}
\newcommand{\Ht}{H}
\newcommand{\At}{A}
\newcommand{\Yt}{Y}
\newcommand{\tQtp}{\widetilde{Q}^*(m_1,m_2)}
\newcommand{\tQtpzero}{\widetilde{Q}^*(p^{m_1}, p^0)}
\newcommand{\tQtpone}{\widetilde{Q}^*(p^{m_1}, p^1)}
\newcommand{\tQtptwo}{\widetilde{Q}^*(p^{m_1}, p^2)}

\kommentar {
\newtheorem{lemma}{Lemma}
\newtheorem{theorem}[lemma]{Theorem}

\newtheorem{conjecture}[lemma]{Conjecture}
\numberwithin{equation}{section} \numberwithin{lemma}{section} }

\def\thebibliography#1{\section{\centerline{\sc References}}
  \global\def\@listi{\leftmargin\leftmargini
               \labelwidth\leftmargini \advance\labelwidth-\labelsep
               \topsep 1pt plus 2pt minus 1pt
               \parsep 0.25ex plus 1pt \itemsep 0.25ex plus 1pt}
  \list {[\arabic{enumi}]}{\settowidth\labelwidth{[#1]}\leftmargin\labelwidth
    \advance\leftmargin\labelsep\usecounter{enumi}}
    \def\newblock{\hskip .11em plus .33em minus -.07em}
    \sloppy
    \sfcode`\.=1000\relax}

\begin{document}

\title[One-level density of families of elliptic curves]{One-level density of families of elliptic curves and the Ratios Conjectures}

\vspace {2 in}

\author[David]{C. David}\email{cdavid@mathstat.concordia.ca}
\address{Department of Mathematics and Statistics, Concordia University, Montrï¿½al, QC, H3G 1M8, Canada}

\author[Huynh]{D.K. Huynh}\email{dkhuynhms@gmail.com}
\address{Department of Pure Mathematics, University of Waterloo, Waterloo, ON, N2L 3G1, Canada}

\author[Parks]{J. Parks}\email{jparks@mathstat.concordia.ca}
\address{Department of Mathematics and Statistics, Concordia University, Montrï¿½al, QC, H3G 1M8, Canada}



\date{\today}


\begin{abstract}
Using the ratios conjectures as introduced by Conrey, Farmer and Zirnbauer \cite{CFZ},
we obtain closed formulas for the one-level density
for two families of L-functions attached to elliptic curves.
From those closed formulas, we can determine the underlying symmetry types of the families. The one-level scaling density for the first family corresponds to the orthogonal distribution as predicted by the conjectures of Katz and Sarnak,
and the one-level scaling density for the second family
is the sum of the Dirac distribution and the even orthogonal distribution.
This seems to be a new phenomenon, caused by the fact that the
curves of the second families have odd rank. Then, there is a trivial zero at
the central point which accounts for the Dirac distribution, and also affects the
remaining part of the scaling density which is then (maybe surprisingly) the even orthogonal distribution. The one-level density for this second family was
studied in the past for test functions with Fourier transforms of small support, but
since the Fourier transforms of the even orthogonal and odd orthogonal distributions
are undistinguishable for small support, it was not possible to identify the distribution with those techniques. This can be done with the ratios conjectures, and it sheds more light on ``independent" and ``non-independent" zeroes and the repulsion
phenomenon.
\end{abstract}

\maketitle 


\section{Introduction}

Since the work of Montgomery \cite{mont} on the pair correlation of the zeroes of the
Riemann zeta function, it is known that there are many striking similarities between
the statistics attached to zeroes of L-functions and eigenvalues of random
matrices. The seminal work of Montgomery was extended and generalised in many
directions, in particular to the study of statistics of zeroes in families of
L-functions, and their relation to the distribution laws for eigenvalues of random matrices.
It is predicted by the Katz and Sarnak philosophy that in the limit
(for large conductor), the
statistics for the zeroes in families of L-functions follow distribution laws
of random matrices.

We consider in this paper the one-level density for two families of L-functions attached to elliptic curves. Let
$\mathcal{F}$ be such a family of elliptic curves, and let
\begin{equation}
\mathcal{F}(X) = \left\{ E \in \mathcal{F} \;:\; N_E \leq X \right\}.
\end{equation}
be the set of curves of conductor $N_E$ bounded by $X$.

For each $E \in \mathcal{F}$, its
one-level
density is the smooth counting function
\begin{equation}
D(E, \phi) = \sum_{\gamma_E} \phi(\gamma_E),
\end{equation}
where the sum runs over the imaginary part of the normalised zeroes $\gamma_E$ of the L-function $L(s,E)$ of the curve $E$.
We assume that the Generalized Riemann Hypothesis holds for the L-functions $L(s,E)$ which are
normalised such that we can write the zeroes in the critical strip as $\rho_E = 1/2 + i\gamma_E $ with
$\gamma_E \in \RR$ (see Section \ref{allEC} for details). Furthermore, $\phi$ is an even Schwartz test
function.

The average of the one-level density over the family $\mathcal{F}(X)$ is then
defined as
$$
D({\mathcal{F}}; \phi, X) := \frac{1}{|\FX|} \sum_{E \in \FX}
D(E, \phi).
$$

Katz and Sarnak predicted that the average one-level density should satisfy
\begin{equation}
\lim_{X \rightarrow \infty} D({\mathcal{F}}; \phi, X) =
\int_{-\infty}^\infty \phi(t) \mathcal{W}(t) \, dt,
\end{equation}
where $\mathcal{W}(G)$ is the one-level scaling density of eigenvalues near 1 in the group of
random matrices corresponding to the symmetry type of the family $\mathcal{F}$.
Remarkably, it is believed that all natural families can be described by very few
symmetry types, namely we have
\begin{equation} \label{symmetrytypes}
\mathcal{W}(G)(t) = \left\{ \begin{array}{ll} 1 & \mbox{if $G=U$;} \\
1 - \frac{\sin{2 \pi t}}{2 \pi t} & \mbox{if $G = \mbox{Sp}$;}\\
1 + \frac{1}{2} \delta_0(t) & \mbox{if $G = O$;}\\
1 + \frac{\sin{2 \pi t}}{2 \pi t} & \mbox{if $G = SO(\mbox{even})$;}\\
1 + \delta_0(t) - \frac{\sin{2 \pi t}}{2 \pi t} & \mbox{if $G = SO(\mbox{odd})$;}\end{array} \right.
\end{equation}
where $\delta_0$ is the Dirac distribution, and $U, \mbox{Sp}, O, SO(\mbox{even}), SO(\mbox{odd})$, are  the groups
of unitary, symplectic, orthogonal, even orthogonal and odd orthogonal matrices respectively. The function $\mathcal{W}(G)(t)$ is called the one-level scaling density of the group $G$.
We refer the reader to \cite{katz-sarnak-book} for details.

There has been extensive research dedicated to gathering evidence for the Katz and Sarnak conjecture
for the one-level density for various families in the last few years. A standard approach is to compute the one-level
density for test functions $\phi$ with limited support of the Fourier transform, i.e.
$\mbox{supp} \; \hat{\phi} \subseteq (-a, a)$  for some $a \in \RR$. In order to distinguish
between the symmetry types of \eqref{symmetrytypes}, one needs to prove results for
a test function $\phi$ with Fourier transform supported outside $[-1,1]$. This approach
was used in many papers, including \cite{ILS} for various families, and \cite{Young2006}
for the families of elliptic curves over $\QQ$ with conductor up to $X$. As in \cite{Young2006}, we are considering
families of elliptic curves,  but
using a different approach to
study the one-level density, by first writing the {\it ratios conjecture}
of \cite{CFZ}
corresponding to our family,
and using it to deduce a formula for the one-level density.
This allows us to verify (under the ratios conjectures) the symmetry type of the family,
and also to obtain lower order terms for the distribution $\mathcal{W}(G)(t)$.
The arithmetic differences between families with same symmetry type are then reflected in the lower order terms of their one-level densities.
This
approach was also used by the second author in his Ph.D. thesis for the family
of quadratic twists of a fixed elliptic curve \cite{DK-thesis, HKS}.
These lower order terms could also help to understand the average behavior of low-lying zeros (i.e., zeros near or at the central point) of relatively small conductor.
Some other applications to the ratios conjectures to study the one-level densities in various families can be found
in \cite{conrey-snaith, HKS}.

The L-function ratios conjectures originated from the work of Farmer \cite{farmer} about
shifted moments of the Riemann zeta function, and the work of Nonnenmacher and Zirnbauer
\cite{NZ}
about the ratios of characteristic polynomials of random matrices. This was then used
to obtain conjectures about shifted ratios of L-functions from number theory by
Conrey, Farmer and Zirnbauer \cite{CFZ}.

We first consider the family of all elliptic curves over $\QQ$ with discriminant of size about $X$ in Section \ref{allEC}.
The conjectural closed formula for the one-level density for this family is given in
Theorem \ref{oneleveldensityresult}, and we can see that the underlying symmetry type
is orthogonal and then match the conjectures of Katz and Sarnak for the family of all
elliptic curves. For this family, evidence for the symmetry type  also follows from the work of \cite{Young2006} for test functions of limited support.

We then study in Section \ref{oneparameter} a one-parameter family of elliptic curves which was first considered
by Washington \cite{washington}, namely the family
$$
E_t : y^2 = x^3 + t x^2 - (t+3)x + 1, \;\;t \in \ZZ.$$
It was shown by Washington \cite{washington} (under some hypothesis for a positive proportion of $t \in \ZZ$), and then by Rizzo
\cite{rizzo} (unconditionally
for all $t \in \ZZ$) that the sign of the functional equation is negative for the
L-functions of the curves $E_t$, which makes is a very interesting family.
The conjectural closed formula for the one-level density for this family
is given in
Theorem \ref{tratiostheoremtwo}, and the leading terms for the one-level scaling density are
${\mathcal{W}}(\tau) = 1 + \delta_0(\tau) + \frac{\sin{2 \pi \tau}}{2 \pi \tau}$. This
does not correspond to any of the densities
${\mathcal{W}}(G)(\tau)$ of \eqref{symmetrytypes}, but it is the sum of two densities,
the Dirac distribution and ${\mathcal{W}}(SO(\mbox{even}))(\tau)$ which usually
corresponds to families of {\it{even}} rank, and we have a family of {\it{odd}}
rank. This can be explained by the special behavior of the zero of the L-functions
$L(s, E_t)$ at $s=1/2$: since the sign of the functional equation is odd,
$L(1/2, E_t)=0$ for all $t \in \ZZ$, but this zero behaves in a special
way. In particular, this zero is independent in the limit from the other zeroes,
and does not cause any repulsion. This phenomenon was also studied by Miller \cite{Miller-investigation} for general one-parameter families of rank $r$.
Then, by Silverman's specialization theorem \cite{silverman}, every curve in the
family have rank at least $r$, and the $r$ forced zeroes (from the
Birch and Swinnerton-Dyer conjecture) are called the {\it{family zeroes}}.
It was noticed by Miller, by computing the one-level density for test functions
$\phi$ with Fourier transform $\hat{\phi}$ of limited support, that those zeroes are independent from the remaining zeroes, and should correspond to a sum of $r$ Dirac
functions in the density $\mathcal{W}(\tau)$ of the family. But the density function
could not be completely determined (even for the case $r=1$ that we
are considering here)
because one can only take limited support for $\hat{\phi}$, and this
does not allow us to differentiate $\mathcal{W}(G)(\tau)$ between $G=SO(\mbox{odd})$ and $G=SO(\mbox{even})$. See \cite[Section 6.1.3]{Miller}.

By using the ratios conjectures, we obtain that once removing the contribution $\delta_0(\tau)$ coming from
from the family zero, the one-level scaling density is $\mathcal{W}(SO(\mbox{even}))(\tau)$. We give in
Section \ref{weird-density-explained} some details explaining why once the
zero is removed, the corresponding L-functions should indeed behave like a family of even rank, according to the Birch and Swinnerton-Dyer conjectures.\\

\kommentar{
\ccom{We are getting that the sign of the FE is always 1 because we are restricting
to $t \in \ZZ$. It would not be possible for $t \in \QQ$. Also, this can only
happen when the reduction is multiplicative everywhere, see the work of Helfgott.
Maybe discuss that}}

\kommentar{
Finally, we  remark that
an interesting feature of the formulas for the one-level density deduced with the
ratios conjecture approach is that one can get the leading term verifying the underlying symmetry, but also lower order terms, which could be used to refine experimental statistics for small conductor. The first few lower order terms for our families are displayed
in Theorem \ref{oneleveldensityresult} and Theorem \ref{tratiostheoremtwo}.}

\noindent{\bf Acknowledgements:}
The authors are very grateful to  Sandro Bettin, Steve J. Miller, Mike Rubinstein, Nina Snaith and Matt Young for many helpful discussions
related to this paper. A large part of this work was done when the second author was visiting the CRM
and Concordia University, and he thanks both institutions for their hospitality.

\section{The family of all elliptic curves}
\label{allEC}

Let $E_{a, b}$ be an elliptic curve over $\QQ$ given
by
\begin{equation} \label{weierstrassellipticcurve}
\Eab: y^2 = x^3 + ax + b.
\end{equation}
We fix some integers $(r,t)$ such that $(r,3)=1$ and $(t, 2)=1$. We will use them
to impose congruences modulo 6 on $a,b$ to insure that $\Eab$ is minimal at $p=2, 3$, so
we remark that there are 12 choices of $(r,s)$.

We study the family
 \begin{equation*}
       \FX = \{E = \Eab : a\equiv r \modd{6}, b\equiv t\modd{6}, |a| \leq X^{\frac{1}{3}}, |b| \leq X^{\frac{1}{2}}, p^4\mid a \Rightarrow p^6 \nmid b \}
 \end{equation*}
of all elliptic curves having  discriminant of size $\asymp X$. The conditions on $\Eab \in \FX$ insure that
$\Eab$ is a minimal model at all primes $p$. \\

Let $L(s, E)$ denote the $L$-function attached to $E$, normalised in such a way that
the center of the critical strip is the line $\mbox{Re}(s)=1/2$. The average one-level
density over the family is then
$$
D({\mathcal{F}}; \phi, X) = \frac{1}{|\FX|} \sum_{E \in \FX}
\sum_{\gamma_E}\phi(\gamma_E),
$$
where $\gamma_E$ runs over the ordinates of the non-trivial zeroes of $L(s,E)$.

By Cauchy's theorem, we can write the average one-level density as
\begin{equation} \label{eq:Cauchy}
D(\F; \phi, X) = \frac{1}{|\FX|} \sum_{E \in \FX} \frac{1}{2\pi
i}\left(\int_{(c)} - \int_{(1-c)}\right)\frac{L'(s,E)}{L(s,E)}
\phi(-i(s-1/2)) ds
\end{equation}
with $\tfrac{1}{2}<c<1$.

Our strategy is to use the ratios conjectures to write a closed
formula for the logarithmic derivative of $L(s,E)$ in \eqref{eq:Cauchy}.
Following the approach of \cite{CFZ}, we consider the ratio
\begin{equation} \label{eq:ratio}
\frac{1}{|\FX|} \sum_{E \in \FX} \frac{L(\frac{1}{2} +
\alpha,E)}{L(\frac{1}{2} + \gamma,E)}
\end{equation}
for $\al, \ga \in \CC$ with $\RRe (\al),
\RRe (\ga) > 0$.

For a minimal model $E = E_{a,b}$, we have that $\lambda_E(n) = \lambda_{a,b}(n)$
where for $p\neq 2$, $\lab(p)$ is given by
\begin{equation*}
\lab(p) =\frac{1}{\sqrt{p}}(p+1-\#E_p(\FF_p))= -\frac{1}{\sqrt{p}} \sum_{x \modd p} \left(\frac{x^3 + ax
+ b}{p} \right)
\end{equation*}
where $\displaystyle \left( \frac{\cdot}{p} \right)$ denotes the Legendre symbol. If
$p=2$ then $\eqref{weierstrassellipticcurve}$ has a cusp and
$\lambda_{a,b}(2^k)=0$ for all $k\geq 1.$

We recall that the L-function attached to an elliptic curve $E$ is given by
\begin{equation} \label{multiplicative}
L(s, E) = \sum_{n = 1}^\infty \frac{\lambda_E(n)}{n^s} = \prod_p
\left(1- \frac{\lambda_E(p)}{p^s} + \frac{\psi_N(p)}{p^{2s}} \
\right)^{-1}
\end{equation}
where $\psi_N$ is the principal Dirichlet character modulo the
conductor $N_E$ of $E$, i.e.,
\begin{equation*}
\psi_N(p) = \begin{cases}
1 &\:{\rm if}\:p\nmid N_E,\\
0 &\:{\rm if}\:p|N_E.
\end{cases}
\end{equation*}

It follows from $\eqref{multiplicative}$ that
$\lambda_E(n)$ is multiplicative, and prime powers can be computed by
\begin{equation} \label{heckerelations}
\lambda_{E}(p^j) = \bigg\{\begin{array}{ll}
U_j\left(\frac{\lambda_{E}(p)}{2}\right) & \qquad {\rm if}\:
(p,N_E)=1, \\
\lambda_{E}^j(p) & \qquad {\rm if}\: (p,N_E)>1,
\end{array}\bigg.
\end{equation}
where $U_j(x)$ is the $j$-th Chebyshev polynomial of the second
kind. The definition of the Chebyshev polynomials and their properties will be given shortly.

It was proven by Wiles \textit{et al} \cite{wiles, taylor-wiles, bcdt} that
\begin{equation*} \Lambda(s,E) = \Gamma(s+1/2) \left( \frac{\sqrt{N_E}}{2 \pi} \right)^s L(s, E)
\end{equation*} satisfies the functional equation
\begin{equation*}
 \Lambda(s,E)  = \omega_E
\Lambda(1-s, E)
\end{equation*}
where $\omega_E =\pm 1 $ is called the root number of $E$. It follows that we can write the values
$L(s,E)$ as
\begin{equation} \label{AFE}
L(s,E) = \sum_n \frac{\lambda_E(n)}{n^{s}} V_s \left( \frac{2 \pi n}{Y \sqrt{N_E}} \right) + \omega_E X_E(s) \sum_n
\frac{\lambda_E(n)}{n^{1-s}} V_{1-s} \left( \frac{2 \pi n Y}{\sqrt{N_E}} \right),
\end{equation}
where
\begin{equation}
\label{definitionofXE}
X_E(s) =
\frac{\Gamma(\frac{3}{2}-s)}{\Gamma(\frac{1}{2}+s)}\left(\frac{\sqrt{N_E}}{2\pi}\right)^{1-2s}
\end{equation}
and $V_s(y)$ is a smooth function which decays rapidly for large values of $y$.
The above identity is called the approximate functional equation for $L(s,E)$, and
we refer the reader to \cite[Theorem 5.3]{I-K-book} for the details.

One of the steps in the recipe leading to the ratios conjectures is to use
the two sums of the approximate functional equation \eqref{AFE} at $s = \frac{1}{2} + \alpha$ ignoring questions of convergence, or error terms, i.e.
the ``principal sum"
\begin{equation} \label{eq:approximate-princ}
\sum_n \frac{\lab(n)}{n^{\frac{1}{2} + \al}}
\end{equation}
and the ``dual sum"
\begin{equation} \label{eq:approximate-dual}
\omega_E X_E(\tfrac{1}{2} + \al) \sum_n
\frac{\lab(n)}{n^{\frac{1}{2}-\al}} \end{equation}

Finally, we write
\begin{equation} \label{eq:Moebius}
\frac{1}{L(s,\Eab)} = \prod_p \left(1 - \frac{\lab(p)}{p^s} +
\frac{\psi_N(p)}{p^{2s}} \right) = \sum_{n=1}^\infty
\frac{\mu_{a,b}(n)}{n^s},
\end{equation}
where $\mu_{a,b}$ is a multiplicative function given by
\begin{equation} \label{definitionofmu}
\mu_{a,b}(p^k) =
\begin{cases}
-\lab(p) & \mbox{~if~} k = 1,\\
\psi_N(p) & \mbox{~if~} k = 2,\\
0 & \mbox{~if~} k > 2.
\end{cases}
\end{equation}

We are now ready to derive the L-function ratios conjecture for our family. Following the
standard recipe from \cite{CFZ} (see also \cite{conrey-snaith, HKS} for applications to
other families), we replace the
numerator of \reff{eq:ratio} with the principal sum \reff{eq:approximate-princ}
and the dual sum \reff{eq:approximate-dual} of the approximate functional equation and the
denominator of \reff{eq:ratio} with \reff{eq:Moebius}. We first focus on
the principal sum which gives the
sum
\begin{equation} \label{eq:ratiostepone}
R_1(\alpha, \gamma):=\frac{1}{|\FX|}\sum_{E_{a,b} \in \FX} \sum_{m_1, m_2}
\frac{\lab(m_1)\mu_{a,b}(m_2)}{m_1^{\frac{1}{2} +
\al}m_2^{\frac{1}{2} + \ga}}.
\end{equation}
We will consider in a second step the sum coming from
the dual sum, namely the sum
\begin{equation} \label{eq:ratiostepone-dual}
R_2(\alpha, \gamma):=  \frac{1}{|\FX|} \sum_{E_{a,b} \in \FX}\omega_{E_{a,b}} X_{E_{a,b}} \left(\frac{1}{2}+\al\right) \sum_{m_1, m_2}
\frac{\lab(m_1)\mu_{a,b}(m_2)}{m_1^{\frac{1}{2} -
\al}m_2^{\frac{1}{2} + \ga}}.
\end{equation}


\newcommand{\tQm}{\widetilde{Q}^*(m_1,m_2)}
\newcommand{\tQp}{\widetilde{Q}^*(p^{m_1},p^{m_2})}
\newcommand{\tQpmtwoisone}{\widetilde{Q}^*(p^{m_1},p^1)}

The first step to obtain the ratios conjecture for our family is now to replace each $(m_1, m_2)$-summand in
\eqref{eq:ratiostepone} by its
average
\begin{equation} \label{average-ratio}
\lim_{X\rightarrow \infty}\frac{1}{|\FX|}\sum_{E_{a,b} \in
\FX}\lab(m_1)\mu_{a,b}(m_2)
\end{equation}
over all curves in the family. This is similar to the work of Young in \cite{Young2010}
where the author makes a conjecture on the moments of the central values $L(1/2,E)$ for the
same family $\mathcal{F}$. He is then led to averages of the type
\begin{equation*}
\lim_{X\rightarrow \infty}\frac{1}{|\FX|}\sum_{E_{a,b} \in
\FX}\lab(m_1) \cdots \lab(m_k)
\end{equation*}
for the $k$-th moment. In the following, we will use some of the results of \cite{Young2010}, and redo some of his computations
in our setting for the sake of completeness.

\begin{lemma} \label{average}
Let
\begin{equation} \label{Qsum}
\tQm := \frac{1}{(m^*)^2} \sum_{a, b \modd m^*}
\lab(m_1)\mu_{a,b}(m_2)
\end{equation}
 and let $m^*$ be the product of primes dividing $m=[m_1,m_2]$. Furthermore, set $m_i=\ell_i n_i$ where $(n_i,6)=1$ and $p \mid \ell_i \Rightarrow p=2,3$, and set
\begin{equation*}
\widetilde{Q}^*_{r,t}(m_1,m_2):=\lambda_{r,t}(m_1)\mu_{r,t}(m_2)\widetilde{Q}^*(n_1,n_2)\prod_{\substack{p\mid m\\ p > 3}} (1-p^{-10})^{-1}.
\end{equation*}
Then
\begin{equation} \label{BigRLimit}
\lim_{X\rightarrow \infty}\frac{1}{|\FX|}\sum_{E_{a,b} \in
\FX}\lab(m_1)\mu_{a,b}(m_2)  \sim \widetilde{Q}^*_{r,t}(m_1,m_2).
\end{equation}
Furthermore, $\widetilde{Q}^*_{r,t}$ is multiplicative in $m_1$ and
$m_2$.
\end{lemma}

\begin{proof} We will follow the proof of Lemma 3.2 in \cite{Young2010}. We have that
\begin{align*}
&\lim_{X\rightarrow \infty}\frac{1}{|\FX|}\sum_{E_{a,b} \in
\FX}\lab(m_1)\mu_{a,b}(m_2)\\
=&\lim_{X\rightarrow \infty}\frac{1}{|\FX|}\sum_{\substack{|a|\leq X^{\frac{1}{3}}, |b|\leq X^{\frac{1}{2}} \\ p^4\mid a \Rightarrow p^6\nmid b \\ a\equiv r\modd{6}, b\equiv t\modd{6}}}\lab(m_1)\mu_{a,b}(m_2).
\end{align*}
We now need to extend the definition of $\lab$ and $\mu_{a,b}$ for non-minimal curves $E_{a,b}$. We define
\begin{align*}
\lab(p)&:=\begin{cases}
\lambda_E(p) & \mbox{~if~} E_{a,b}\:{\rm is\:minimal\:at}\:p,\\
0 & \mbox{~otherwise}.
\end{cases}\\
\psi_{a,b}(p)&:=\begin{cases}
1, & \mbox{~if~} p \nmid -16(4a^3+27b^2)),\\
0, & \mbox{~otherwise}.
\end{cases}
\end{align*}
This defines $\mu_{a,b}$ at prime powers by $\eqref{definitionofmu}$, and $\lambda_{a,b}$ is defined at prime powers by the usual relation $\eqref{heckerelations}$. We then extend to $\lambda_{a,b}(n), \mu_{a.b}(n)$ by multiplicativity.

We also have the usual power detector
\begin{equation*}
\sum_{\substack{d^4 \mid a \\ d^6 \mid b}}\mu(d)=\begin{cases} 1, & \quad {\rm if\:there\:does\:not\:exist\:a}\:p\:{\rm such\:that}\:p^4\mid a\:{\rm and}\:p^6\mid b, \\ 0, & \quad {\rm otherwise}.\end{cases}
\end{equation*}
Thus
\begin{align*}
\lim_{X\rightarrow \infty}\frac{1}{|\FX|}\sum_{E_{a,b} \in
\FX}\lab(m_1)\mu_{a,b}(m_2) \end{align*}
can be rewritten as
\begin{align}
\lim_{X\rightarrow \infty}\frac{1}{|\FX|}\sum_{\substack{d\leq X^{\frac{1}{12}} \\ (d,6)=1}}\mu(d)\sum_{\substack{|a|\leq d^{-4}X^{\frac{1}{3}}, |b|\leq d^{-6}X^{\frac{1}{2}} \\ a\equiv d^{-4}r\modd{6}, b\equiv d^{-6}t\modd{6}}}\lambda_{ad^4,bd^6}(m_1)\mu_{ad^4,bd^6}(m_2) \label{rrt}
\end{align}
It follows from our definition of $\lab$ and $\mu_{a,b}$ at non-minimal curves that
$$\lambda_{d^4a,d^6b}(n)=\begin{cases} \lambda_{a,b}(n) & \mbox{~if~} (n,d)=1,\\ 0 & \mbox{~otherwise}, \end{cases}$$ and similarly
$$\mu_{d^4a,d^6b}(n)=\begin{cases} \mu_{a,b}(n) & \mbox{~if~} (n,d)=1,\\ 0 & \mbox{~otherwise}. \end{cases}$$
Thus, if $(n,d)=1$, we have $$\lambda_{d^4a,d^6b}(m_1)\mu_{d^4a,d^6b}(m_2)=\lambda_{r,t}(\ell_1)\mu_{r,t}(\ell_2)\lambda_{a,b}(n_1)\mu_{a.b}(n_2),$$
and $\eqref{rrt}$ becomes
\begin{equation*}
\lim_{X\rightarrow \infty}\frac{1}{|\FX|}\lambda_{r,t}(\ell_1)\mu_{r,t}(\ell_2)\sum_{\substack{d\leq X^{\frac{1}{12}} \\ (d,6n)=1}} \mu(d)\sum_{\substack{|a|\leq d^{-4}X^{\frac{1}{3}}, |b|\leq d^{-6}X^{\frac{1}{2}} \\ a \equiv d^{-4}r \modd{6} \\ b\equiv d^{-6}t \modd{6}}} \lambda_{a,b}(n_1)\mu_{a,b}(n_2).
\end{equation*}
Now $\lambda_{a,b}(n_1)\mu_{a,b}(n_2)$ is periodic in $a$ and $b$ with period equal to the product of primes dividing the least common multiple of $n_1,n_2$ say $n^*$. Breaking up the sum over $a$ and $b$ into arithmetic progressions modulo 6, we rewrite the last equation as
\begin{align*}
&\lim_{X\rightarrow \infty}\frac{1}{|\FX|}\lambda_{r,t}(\ell_1)\mu_{r,t}(\ell_2)\frac{4X^{\frac{5}{6}}}{36}\sum_{\substack{d\leq X^{\frac{1}{12}} \\ (d,6n)=1}} \frac{\mu(d)}{d^{10}}\frac{1}{(m^*)^2}\sum_{\substack{\alpha \modd{n^*} \\ \beta \modd{n^*}}}\lambda_{\alpha,\beta}(n_1)\mu_{\alpha,\beta}(n_2) \\
&=\lim_{X\rightarrow \infty}\frac{1}{|\FX|}\lambda_{r,t}(\ell_1)\mu_{r,t}(\ell_2)\frac{X^{\frac{5}{6}}}{9\zeta_{6n}(10)}\widetilde{Q}^*(n_1,n_2),
\end{align*}
where we define
\begin{equation*}
\zeta_m(s):=\prod_{p\mid m}\left(1-\frac{1}{p^s}\right)^{-1}.
\end{equation*}

Since $$|\FX|\sim \frac{X^{\frac{5}{6}}}{9\zeta_{6n}(10)},$$  $\eqref{rrt}$ becomes
\begin{equation*}
\lambda_{r,t}(\ell_1)\mu_{r,t}(\ell_2)\widetilde{Q}^*(n_1,n_2)\frac{\zeta_6(10)}{\zeta_{6n}(10)}.
\end{equation*}
Since $$\frac{\zeta_6(10)}{\zeta_{6n}(10)}=\prod_{\substack{p\mid n \\ p>3}}(1-p^{-10})^{-1} = \prod_{\substack{p\mid m \\ p>3}}(1-p^{-10})^{-1},$$ this completes the proof of $\eqref{BigRLimit}$.
\end{proof}

Now we replace each term of \eqref{eq:ratiostepone} by its average value $\tQprt(m_1,m_2)$, and using
Lemma \ref{average}, we are led to consider
\begin{align}
H(\alpha, \gamma) &:= \sum_{m_1, m_2} \frac{\tQprt(m_1,m_2)}{m_1^{\frac{1}{2} +
\al}m_2^{\frac{1}{2} + \ga}} = \prod_p \sum_{m_1, m_2}
\frac{\tQprt(p^{m_1},p^{m_2})}{p^{m_1(\frac{1}{2} + \al) + m_2(\frac{1}{2} + \ga)}} \nonumber \\
&=\left(\prod_{2\leq p\leq3}\sum_{m_1,m_2}\frac{\lambda_{r,t}(p^{m_1})\mu_{r,t}(p^{m_2})}{p^{m_1(\frac{1}{2}+\alpha)+m_2(\frac{1}{2}+\gamma)}}\right)\left(\prod_{p>3} \sum_{m_1, m_2}
\frac{\delta(p)\tQp}{p^{m_1(\frac{1}{2} + \al) + m_2(\frac{1}{2} + \ga)}}\right),
\label{eq:Qstart}
\end{align}
where $\delta(p)=(1-p^{-10})^{-1}$ if $m_1+m_2>0$ and $\delta(p)=1$ otherwise.

Thus, it suffices to consider $\tQprt(p^{m_1},p^{m_2})$ at a prime $p$ and integers $m_1, m_2$. Notice that we switched notation, and we are now using $m_1, m_2$ for the exponents of the prime powers. By the definition of the Moebius function in \reff{definitionofmu} only
the terms with $m_2 = 0,1$ and $2$ in \reff{eq:Qstart} contribute.
For $p=2,3$, we denote by $E_p(\alpha, \gamma)$ the Euler factor
\begin{eqnarray} \label{eulerfactor}
E_p(\alpha, \gamma):=\sum_{m_1,m_2}\frac{\lambda_{r,t}(p^{m_1})\mu_{r,t}(p^{m_2})}{p^{m_1(\frac{1}{2}+\alpha)+m_2(\frac{1}{2}+\gamma)}}
\end{eqnarray}
at $p$ in $H(\alpha, \gamma)$.

So we have that
\begin{align*}
H(\alpha,\gamma)=&E_2(\alpha,\gamma)E_3(\alpha,\gamma)\prod_{p>3} \sum_{m_1, m_2}
\frac{\delta(p)\tQp}{p^{m_1(\frac{1}{2} + \al) + m_2(\frac{1}{2} + \ga)}}\\
=&E_2(\alpha,\gamma)E_3(\alpha,\gamma) \prod_{p>3} \Bigg(1+(1-p^{-10})^{-1}\\
\times&\left(\sum_{m_1 \geq 1} \frac{\tQpzero}{\denominatoralpha} + \sum_{m_1
\geq 0}\frac{\tQpone}{p^{m_1(\frac{1}{2} + \alpha) + \frac{1}{2} +
\gamma}} + \sum_{m_1 \geq 0}\frac{\tQptwo}{p^{m_1(\frac{1}{2} +
\alpha) + 1+ 2\gamma}}\right) \Bigg),
\end{align*}
where
\begin{eqnarray} \label{m2is0}
\tQpzero &=& \frac{1}{p^2} \sum_{a, b \modd p} \lab(p^{m_1}),\\  \label{m2is1}
 \tQpone &=& -\frac{1}{p^2} \sum_{a, b \modd p}
\lab(p^{m_1}) \lab(p),\\   \label{m2is2}
\tQptwo &=& \frac{1}{p^2}
\sum_{\substack{a, b \modd p \\ p\nmid N_E}} \lab(p^{m_1}).
\end{eqnarray}

In the following theorem, we write a closed formula for $H(\alpha,\gamma)$ in terms
of the trace of the Hecke operators $T_p$, using the Eichler-Selberg Trace Formula,
following \cite{Young2006} (see Lemma \ref{lemma-fromMY} below). We first need some notation.
Let $Tr_j(p)$ denote the trace of the Hecke operator $T_p$ acting on
the space of weight $j$ holomorphic cusp forms on the full modular
group. The normalized trace $Tr^*_j(p)$ is given by
\begin{equation} \label{normalizedtrace}
Tr^*_j(p) = p^{(1-j)/2}Tr_j(p).
\end{equation}
We recall that we have that $Tr_{j}^*(p)=0$ for $j<12$.

Now $H(\alpha,\gamma)$ can be rewritten in terms of $Tr_{j}^*(p)$.

\begin{theorem} \label{factorH} Let $\alpha, \gamma \in \CC$ such that $\RRe(\alpha), \RRe(\gamma) > 0$, and
let $H$ be given by $\eqref{eq:Qstart}$. Then
\begin{align*}
&H(\alpha,\gamma)= E_2( \alpha, \gamma) E_3( \alpha, \gamma) \prod_{p>3}\bigg[1+\left(1-\frac{p^9-1}{p^{10}-1}\right)\bigg(\frac{1}{p^{1+2\gamma}}-
\frac{1}{p^{1+\alpha+\gamma}}\nonumber \\
&+\frac{p^{-(2+\alpha+\gamma)}-p^{-(2+2\gamma)}}{p^{2+2\alpha}-1}+
\left(\frac{p^{1+2\alpha+\gamma}-p^{1+\alpha+2\gamma}+p^{\gamma}-p^{\alpha}}
{p^{\frac{3}{2}+\alpha+2\gamma}}\right)
\summtenisevenpositiv\frac{Tr_{m_1+2}^*(p)}{p^{m_1(\frac{1}{2}+\alpha)}}\bigg)\bigg].
\end{align*}

Furthermore, $H$ has the form
\begin{equation*} \label{hzetaform}
H(\alpha, \gamma) = \frac{\zeta(1 + 2\gamma)}{\zeta(1 + \alpha +
\gamma)} A(\alpha, \gamma)
\end{equation*}
where $A(\alpha,\gamma)$ is holomorphic and non-zero for
$\RRe(\alpha), \RRe(\gamma) > -1/4$. We also have that $A(r,r)=1$ in this region.
\end{theorem}

Before proving Theorem \ref{factorH}, we make some observations and
state some useful lemmata. First we consider the Chebyshev polynomials $U_n(x)$ appearing in
\eqref{heckerelations}
and their properties.
The polynomials $U_n(x)$ satisfy the recursion formula
\begin{eqnarray} \label{recursionCP}
U_{n+2}(x) - 2 x U_{n+1}(x) + U_n(x) = 0, \quad {\rm for}\:n \geq 0,
\end{eqnarray}
which is equivalent to the formal identity
\begin{eqnarray} \label{CP1}
\sum_{n \geq 0} U_{n}(x) t^n = \frac{1}{1 - 2 x t + t^2}.
\end{eqnarray}
Then the first few Chebyshev polynomials are $U_0(x)=1$, $U_1(x)=2x$, $U_2(x)= 4x^2-1$, $U_3(x) = 8x^3 - 4x$, etc. Also, the  Chebyshev polynomials
satisfy
\begin{eqnarray} \label{odd-even}
U_n(-x) = (-1)^n U_n(x),\end{eqnarray}
i.e., $U_n(x)$ is odd when $n$ is odd, and even when $n$ is even. We also define the coefficients $c_{\ell}(m_1, m_2)$ by
\begin{eqnarray*} U_{m_1}(x) U_{m_2}(x) =
\sum_{\ell \geq 0} c_{\ell}(m_1, m_2) U_\ell(x).\end{eqnarray*}

From the properties of the Chebyshev polynomials above we have that if $m_1 + m_2$ is odd (and $p > 2$), then $Q^*(p^{m_1}, p^{m_2})$ is 0. We can see this by first making the change of
variables $a=d^2a', b=d^3b'$ where $d$ is a quadratic nonresidue
modulo $p$. Then we have that $\lambda_{a,b}(p)=-\lambda_{a',b'}(p)$. Now if $m_2=0,1$ then
\begin{align*}
\tQ(p^{m_1},p^{m_2})&= \frac{(-1)^{m_2}}{p^2}\sum_{a,b \modd {p}}\lambda_{a,b}(p^{m_1})\lambda_{a,b}^{m_2}(p) \\
 &= \frac{(-1)^{m_2}}{p^2}\sum_{\substack{a,b \modd {p} \\ p \mid N_E}}\lambda_{a,b}^{m_1+m_2}(p) + \frac{(-1)^{m_2}}{p^2}\sum_{\substack{a,b \modd {p} \\ p \nmid N_E}}U_{m_1}\left(\frac{\lambda_{a,b}(p)}{2}\right)\lambda_{a,b}^{m_2}(p)\\
 & = \frac{(-1)^{m_1}}{p^2}\sum_{\substack{a',b' \modd {p} \\ p \mid N_E}}\lambda_{a',b'}^{m_1+m_2}(p)\\
  &+ \frac{(-1)^{m_1}}{p^2}\sum_{\substack{a',b' \modd {p} \\ p \nmid N_E}}U_{m_1}\left(\frac{\lambda_{a',b'}(p)}{2}\right)\lambda_{a',b'}^{m_2}(p)\\
 &=(-1)^{m_1+m_2}\tQ(p^{m_1},p^{m_2})
\end{align*}
where we have used the property \eqref{odd-even} of the Chebyshev polynomials.
If $m_2=2$ then \begin{eqnarray*}
\tQ(p^{m_1},p^2) &=& \frac{1}{p^2}\sum_{\substack{a,b \modd {p} \\ p \nmid N_E}} U_{m_1} \left( \frac{\lambda_{a,b}(p)}{2} \right)=\frac{(-1)^{m_1}}{p^2}\sum_{\substack{a',b' \modd {p} \\ p \nmid N_E}} U_{m_1} \left( \frac{\lambda_{a',b'}(p)}{2} \right)
\\ &=& (-1)^{m_1} \tQ(p^{m_1},p^2). \end{eqnarray*}

Hence, we have for $p > 3$, each Euler factor in $H(\alpha, \gamma)$ can be written as
\begin{equation}
\label{starthere}
 1+\frac{1}{1-p^{-10}}\left(\sum_{\substack{m_1\geq 2 \\ m_1 \:{\rm even}}} \frac{\tQpzero}{\denominatoralpha} + \summoneisodd
\frac{\tQpone}{p^{m_1(\frac{1}{2} + \alpha) + \frac{1}{2} + \gamma}}
+ \summoneiseven \frac{\tQptwo}{p^{m_1(\frac{1}{2} + \alpha) + 1+
2\gamma}}\right).
\end{equation}

We will use the following result from \cite{Young2010}.

\begin{lemma}[Proposition 4.2, \cite{Young2010}] \label{lemma-fromMY}
Let
\begin{equation}
Q^*(p^{m_1}, p^{m_2}) = \frac{1}{p^2} \sum_{a,b \pmod p} \lambda_{a,b}(p^{m_1}) \lambda_{a,b}(p^{m_2}).
\end{equation}
Then for $p > 3$ and $m_1 + m_2$ even and positive, we have
\begin{eqnarray}
Q^*(p^{m_1}, p^{m_2}) &=& c_0(m_1, m_2) \frac{p-1}{p} + \frac{p-1}{p^{2}} p^{-(m_1+m_2)/2} \\
&&- \sum_{\ell \geq 1} c_\ell(m_1, m_2) \left(
\frac{p-1}{p^{3/2}} Tr^*_{\ell + 2}(p)  + \frac{p-1}{p^{2}} p^{-\ell/2} \right) .
\end{eqnarray}
If $m_1+m_2$ is odd, or $p=2$, then $Q^*(p^{m_1}, p^{m_2}) = 0$.
\end{lemma}

\begin{lemma}\label{resultateins} Let $p > 3$ and $m_1 \geq 2$ even. Then
\begin{equation}
\tQpzero =
-\frac{p-1}{p^{3/2}} Tr^*_{m_1+2}(p).
\end{equation}
\end{lemma}
\begin{proof} This follows immediately from \eqref{m2is0} by specializing Lemma \ref{lemma-fromMY} since
$c_\ell(m_1, 0) = 1$ for $\ell=m_1$ and $0$ otherwise. \end{proof}

\begin{lemma}\label{resultatzwei} Let $p > 3$ and $m_1 \geq 1$ odd.
Then, for $m_1 \geq 3$,
\begin{eqnarray*}
\tQpmtwoisone = \frac{p-1}{p^2} p^{-(m_1-1)/2} + \frac{p-1}{p^{3/2}} \left( Tr^*_{m_1+1}(p)+Tr^*_{m_1+3}(p) \right)
\end{eqnarray*}
and
\begin{eqnarray*}
\tilde{Q}^*(p,p) =  \frac{1-p}{p} .
\end{eqnarray*}
\end{lemma}
\begin{proof} From \eqref{m2is1}, we have that $\tQpmtwoisone = -Q^*(p^{m_1}, p),$ and then
it follows immediately by specializing Lemma \ref{lemma-fromMY} that
\begin{align} \nonumber
\tQpmtwoisone & = -c_0(m_1,1) \frac{p-1}{p} + \sum_{\ell \geq 1} c_{\ell}(m_1,1) \left[\frac{p-1}{p^{3/2}} Tr^*_{\ell+2}(p) + \frac{p-1}{p^2}p^{-\ell/2}\right]\\ \label{relation1}
& - \frac{p-1}{p^2} p^{-(m_1+1)/2}.
\end{align}
For $m_1 \geq
1$, the recursion relation \eqref{recursionCP} gives
\begin{align} \nonumber
U_{m_1}(x) U_1(x) = 2x U_{m_1}(x) = U_{m_1 + 1}(x) + U_{m_1 - 1}(x)
\end{align}
and
\begin{equation}
c_{\ell}(m_1, 1) = \begin{cases}
1 \quad\quad\mbox{~for~} \ell = m_1 -1, m_1 + 1,\\
0 \quad\quad \mbox{~otherwise}.
\end{cases}
\end{equation}
Replacing in \eqref{relation1}, this gives the result for $m_1 \geq 3$. For $m_1 = 1$, we also use
the fact that $Tr^*_{4}(p)=0$.
\end{proof}

\begin{lemma}\label{resultatdrei} If $p>3$, and $m_1 \geq 2$ is even, then
\begin{align}
\tQptwo =
-\frac{p-1}{p^{3/2}} Tr^*_{m_1+2}(p) - \frac{p-1}{p^2}p^{-m_1/2}.
\end{align}
Furthermore,
$$\tilde{Q}^*(p^0, p^2) = \frac{(p-1)}{p}.$$
\end{lemma}
\begin{proof} Let $m_1 \geq 2$. From \eqref{m2is2}, we have that
$$
\tQptwo = \tQpzero - \frac{1}{p^2} \sum_{{a,b \mod p} \atop {p \mid N_E}} \lambda_{a,b}(p^m_1).$$
For $p > 3$, one shows that
$$\sum_{{a,b \mod p} \atop {p \mid N_E}} \lambda_{a,b}(p^m_1) = p^{-m_1/2} (p-1)$$
by parameterizing all pairs $(a,b) \in \FF_p^2$ such that $\Delta \equiv 0 \mod p$ (see
the proof of Proposition 4.2 in \cite{Young2006}). The result then follows from Lemma \ref{resultateins}.
If $m_1=0$, then
$$\tilde{Q}^*(p^0, p^2) = \frac{1}{p^2} \sum_{{a,b \mod p} \atop {p \nmid N_E}} 1 = \frac{p(p-1)}{p^2}.$$
\end{proof}

\begin{proof}[Proof of Theorem \ref{factorH}]

Starting from $\eqref{starthere}$, we have
\begin{align*}
H(\alpha,\gamma)&= E_2(\alpha, \gamma) E_3 (\alpha, \gamma) \prod_{p>3} \Bigg[1+(1-p^{-10})^{-1}\Bigg(\summoneisevenpositiv
\frac{\tQpzero}{p^{m_1(\frac{1}{2}+\alpha)}} \\
&+\frac{1}{p^{\frac{1}{2}+\gamma}}\summoneisoddpositiv\frac{\tQpone}{p^{m_1(\frac{1}{2}+\alpha)}}+\frac{\widetilde{Q}^*(1,p^2)}{p^{1+2\gamma}}
+\frac{1}{p^{1+2\gamma}}\summoneisevenpositiv\frac{\tQptwo}{p^{m_1(\frac{1}{2}+\alpha)}}\Bigg)\Bigg].
\end{align*}
Let $p > 3$. We
consider each of the four terms in the Euler factors $E_p(\alpha, \gamma)$ separately. From Lemma \ref{resultateins}, we have that
\begin{equation}
\summoneisevenpositiv \frac{\tQpzero}{p^{m_1(\frac{1}{2}+\alpha)}}=
\frac{-(p-1)}{p^{\frac{3}{2}}}\summoneisevenpositiv
\frac{Tr_{m_1+2}^*(p)}{p^{m_1(\frac{1}{2}+\alpha)}}.
\end{equation}
From Lemma \ref{resultatzwei}, we have that
\begin{eqnarray*}
\frac{1}{p^{1/2 + \gamma}} \sum_{\substack{m_1 \geq 1 \\m_1 \rm odd}}  \frac{\tQpone}{p^{m_1(\frac{1}{2}+\alpha)}}
&=&
\frac{-(p-1)}{p^{2 + \gamma+\alpha}} +
\sum_{\substack{m_1 \geq 3 \\m_1 \rm odd}}  \frac{p-1}{p^{2+\gamma+m_1(1+\alpha)}} \\
&&  + \frac{p-1}{p^{2 + \gamma}} \sum_{\substack{m_1 \geq 9 \\m_1 \rm odd}} \frac{
Tr_{m_1 + 1}^*(p) + Tr_{m_1+3}^*(p)}{p^{m_1(1/2 + \alpha)}} .
\end{eqnarray*}

We compute that
\begin{eqnarray*}
\sum_{\substack{m_1 \geq 3 \\m_1 \rm odd}}  \frac{p-1}{p^{2+\gamma+m_1(1+\alpha)}}
= \frac{p-1}{p^{3 + \alpha+\gamma}} \sum_{m_1 \geq 1} \left( \frac{1}{p^{1+\alpha}} \right)^{2 m_1}
= \frac{p-1}{p^{3 + \alpha+\gamma}(p^{2 + 2 \alpha}-1)}
\end{eqnarray*}

Finally, from Lemma \ref{resultatdrei} we have that
\begin{equation}
\frac{\widetilde{Q}^*(1,p^2)}{p^{1+2\gamma}}=\frac{p-1}{p^{2+2\gamma}}
\end{equation}
and
\begin{eqnarray*}
\summoneisevenpositiv\frac{\tQptwo}{p^{m_1(\frac{1}{2}+\alpha)+1+2\gamma}}&=&
\frac{-(p-1)}{p^{\frac{5}{2}+2\gamma}}
 \summoneisevenpositiv\frac{Tr_{m_1+2}^*(p)}{p^{m_1(\frac{1}{2}+\alpha)}}-\frac{p-1}{p^{3+2\gamma}}
 \summoneisevenpositiv
 \frac{1}{p^{m_1(1+\alpha)}}\\
 &=&\frac{-(p-1)}{p^{\frac{5}{2}+2\gamma}}\summoneisevenpositiv\frac{Tr_{m_1+2}^*(p)}{p^{m_1(\frac{1}{2}+\alpha)}}-
 \frac{p-1}{p^{3+2\gamma}(p^{2+2\alpha} - 1)}.
\end{eqnarray*}

Since $Tr_{j}^*(p)=0$ for $j<12$, summing the four terms above and collecting terms gives
\begin{align*}
H(\alpha,\gamma)&=E_2(\alpha, \gamma) E_3 (\alpha, \gamma)\prod_{p>3} \Bigg[1+(1-p^{-10})^{-1}\Bigg(\frac{p-1}{p^{2+2\gamma}}-\frac{p-1}{p^{2+\alpha+\gamma}}\\
&+\frac{p-1}{p^{3+\gamma} (p^{2 + 2 \alpha} - 1)} \left( \frac{1}{p^\alpha} - \frac{1}{p^\gamma} \right)
-\frac{p-1}{p^{\frac{3}{2}}}\left(\frac{1}{p^{1+2\gamma}}+1\right)
\summtenisevenpositiv\frac{Tr_{m_1+2}^*(p)}{p^{m_1(\frac{1}{2}+\alpha)}}\\
&+\frac{p-1}{p^{2+\gamma}}\summnineisoddpositiv\frac{Tr_{m_1+1}^*(p)+Tr_{m_1+3}^*(p)}{p^{m_1(\frac{1}{2}+\alpha)}}
\Bigg)\Bigg],
\end{align*}
and factoring out $\frac{p-1}{p}$ gives
\begin{align}
H(\alpha,\gamma)&=E_2(\alpha, \gamma) E_3 (\alpha, \gamma)\prod_{p>3} \Bigg[1+\left(1-\frac{p^9-1}{p^{10}-1}\right)\Bigg(\frac{1}{p^{1+2\gamma}}-\frac{1}{p^{1+\alpha+\gamma}}\nonumber \\
&+\frac{1}{p^{2+\gamma} (p^{2 + 2 \alpha} - 1)} \left( \frac{1}{p^\alpha} - \frac{1}{p^\gamma} \right)-\left(\frac{1}{p^{\frac{3}{2}+2\gamma}}+\frac{1}{p^{\frac{1}{2}}}\right)
\summtenisevenpositiv\frac{Tr_{m_1+2}^*(p)}{p^{m_1(\frac{1}{2}+\alpha)}}\nonumber \\
&+\summnineisoddpositiv\frac{Tr_{m_1+1}^*(p)+Tr_{m_1+3}^*(p)}{p^{m_1(\frac{1}{2}+\alpha)+1+\gamma}}
\Bigg) \Bigg].\label{cuteform}
\end{align}

Finally, we compute that
\begin{align*}
&\left(\frac{-1}{p^{\frac{3}{2}+2\gamma}}-\frac{1}{p^{\frac{1}{2}}}\right)
\summtenisevenpositiv\frac{Tr_{m_1+2}^*(p)}{p^{m_1(\frac{1}{2}+\alpha)}}
+\summnineisoddpositiv\frac{Tr_{m_1+1}^*(p)+Tr_{m_1+3}^*(p)}{p^{m_1(\frac{1}{2}+\alpha)+1+\gamma}}
\\
=&\left(\frac{p^{1+2\alpha+\gamma}+p^{\gamma}-p^{\alpha}-p^{1+\alpha+2\gamma}}
{p^{\frac{3}{2}+\alpha+2\gamma}}\right)\summtenisevenpositiv\frac{Tr_{m_1+2}^*(p)}{p^{m_1(\frac{1}{2}+\alpha)}}.
\end{align*}

Now, looking at the contributions of the Euler factors in \eqref{cuteform}, the term
${p^{-(1+2\gamma)}}$
contributes a pole,
and the term
${-p^{-(1+\alpha+\gamma)}}$ contributes a zero, and we
can write
\begin{equation} \label{definitionofA}
H(\alpha,\gamma) = \frac{\zeta(1+2\gamma)}{\zeta(1+\alpha+\gamma)}
A(\alpha,\gamma)
\end{equation}
as required where $A(\alpha, \gamma)$ converges uniformly and absolutely
for $\RRe(\alpha), \RRe(\gamma) > -1/4$ since $|Tr_{j}^*(p)|\leq 1$.

Finally, by setting $\alpha=\gamma=r$, we also have that
\begin{eqnarray}
A(r,r)=H(r,r) = 1,
\end{eqnarray}
and in fact, each of the Euler factors $E_p(r,r)=1$.
If $p > 3$, it can be seen directly by setting $\alpha=\gamma=r$
in
$\eqref{cuteform}$, as
\begin{eqnarray*}
&& \bigg(\frac{1}{p^{1+2r}}-\frac{1}{p^{1+2r}}+\frac{p^{-(2+2r)}-p^{-(2+2r)}}{p^{2+2r}-1} \\
&& \quad\quad +\left(\frac{p^{1+3r}-p^{1+3r}+p^{r}-p^{r}}{p^{\frac{3}{2}+3r}}\right)
\summtenisevenpositiv\frac{Tr_{m_1+2}^*(p)}{p^{m_1(\frac{1}{2}+r)}}\bigg) = 0.
\end{eqnarray*}
For $p=2,3$ we use a more general result, and show directly that for all Euler factors $E_p(\alpha, \beta)$
as given by \eqref{eulerfactor}, that $E_p(r,r)=1$ because of the Hecke relations.
This is done in Lemma \ref{generalEF} of the next section for the Euler factors in the second family, where we have no closed
formula, and it shows the result for $E_2(r, r)$ and $E_3(r, r)$ by adapting the proof for this family.  \end{proof}


Then, by replacing each $(m_1, m_2)$-summand in \reff{eq:ratiostepone} by its average over the family, we
replace $R_1(\alpha, \gamma)$ by
$H(\alpha,\gamma)$ as given by \eqref{definitionofA}.
We also set
\begin{equation} \label{definitionofY}
Y(\alpha,\gamma) :=\frac{\zeta(1+2\gamma)}{\zeta(1+\alpha+\gamma)}.
\end{equation}

We now consider the sum $R_2(\alpha, \gamma)$ of \eqref{eq:ratiostepone-dual} coming from the dual sum of the approximate functional equation.
Working completely similarly, replacing each $(m_1, m_2)$-summand in \reff{eq:ratiostepone} by its average over the family,
we rewrite $R_2(\alpha, \gamma)$ as
\begin{equation}
\omega_E
\left(\frac{\sqrt{N}}{2\pi}\right)^{-2\alpha}\frac{\Gamma(1-\alpha)}{\Gamma(1+\alpha)}
Y(-\alpha,\gamma)A(-\alpha,\gamma).
\end{equation}

Then, replacing each summand in \reff{eq:ratio},  we get  the ratios conjecture for our family.

\begin{conjecture}[Ratios Conjecture] \label{ratiosconjecture}  Let $\varepsilon >0$.
Let $\alpha, \gamma \in \CC$ such that $\RRe(\alpha) > -1/4$, $\RRe(\gamma) \gg \frac{1}{\log{X}}$ and $\IIm(\alpha), \IIm(\gamma) \ll_\varepsilon X^{1-\varepsilon}$. Then
\begin{align*}
& \frac{1}{|\FX|} \sum_{E \in \FX} \frac{L(\frac{1}{2} +
\alpha,E)}{L(\frac{1}{2} + \gamma,E)}\\
& = \frac{1}{|\FX|} \sum_{E \in \FX} \left[Y(\alpha,\gamma)A(\alpha,\gamma) + \omega_E \left(\frac{\sqrt{N_E}}{2\pi}\right)^{-2\alpha}\frac{\Gamma(1-\alpha)}{\Gamma(1+\alpha)} Y(-\alpha,\gamma)A(-\alpha,\gamma) \right] \\
& + O(X^{-1/2+\varepsilon})
\end{align*}
where $Y(\alpha,\gamma)$ is defined in \reff{definitionofY} and
$A(\alpha,\gamma)$ in \reff{definitionofA}.
\end{conjecture}

We remark that the error term  $O(X^{-1/2+\varepsilon})$ is part of the statement of the ratios conjectures, and the power
on $X$ is not suggested by any of the steps used in arriving at the main expression in Conjecture \ref{ratiosconjecture}. The lower bound for $\RRe(\gamma)$ and the upper bound for $\IIm(\alpha), \IIm(\gamma)$ are also part of the statement of the ratios conjectures, and should be thought as reasonable conditions under which the
Conjecture \ref{ratiosconjecture} should hold. For more details, we refer the reader
to \cite{conrey-snaith} (see for example the conditions (2.11b) and (2.11c) on page 6).
Ignoring issues about the error term and uniformity, there should of course be a condition of the type $\RRe(\gamma) \geq \delta$ for some $\delta>0$.

To get the one-level density for our family, we have to differentiate the result of Conjecture \ref{ratiosconjecture}
with respect to $\alpha$ and use \eqref{eq:Cauchy}. We then obtain

\begin{theorem} \label{ratiostheorem} Let $\varepsilon > 0$, and $r \in \CC$. Assuming the Ratios Conjecture \ref{ratiosconjecture},  $\RRe(r) \gg \frac{1}{\log{X}}$ and
$\IIm(r) \ll X^{1-\varepsilon}$,
we have
\begin{align*} \nonumber
& \frac{1}{|\FX|} \sum_{E \in \FX} \frac{L'(\frac{1}{2} + r,E)}{L(\frac{1}{2} + r,E)} = \frac{1}{|\FX|} \sum_{E \in \FX}  \bigg[-\frac{\zeta'}{\zeta}(1+2r) + A_{\alpha}(r,r)\\
 & - \omega_E\left(\frac{\sqrt{N_E}}{2\pi}\right)^{-2r}\frac{\Gamma(1-r)}{\Gamma(1+r)} \zeta(1+2r)A(-r,r) \bigg] + O(X^{-1/2+\varepsilon}),
\end{align*}
where $A_{\alpha}(r,r)$ is defined in \reff{definitionofAone}.
\end{theorem}

\begin{proof}
We set
\begin{equation} \label{definitionofAone}
A_{\alpha}(r,r) := \frac{\partial}{\partial
\alpha}A(\alpha,\gamma)\bigg|_{\alpha=\gamma=r}.
\end{equation}
Then
\begin{equation}\label{ronefirstfam}
\frac{\partial}{\partial \alpha}Y(\alpha, \gamma)A(\alpha,
\gamma)\bigg|_{\alpha=\gamma=r} = -\frac{\zeta'}{\zeta}(1+2r) +
A_{\alpha}(r,r),
\end{equation}
and
\begin{align} \nonumber
\frac{\partial}{\partial \alpha}\bigg[ & \omega_E
\left(\frac{\sqrt{N}}{2\pi}\right)^{-2\alpha}\frac{\Gamma(1-\alpha)}{\Gamma(1+\alpha)}
Y(-\alpha,\gamma)A(-\alpha,\gamma)\bigg]_{\alpha=\gamma=r}
\\ \nonumber
& =
\omega_E\left(\frac{\sqrt{N}}{2\pi}\right)^{-2\alpha}\frac{\Gamma(1-\alpha)}{\Gamma(1+\alpha)}
A(-\alpha,\gamma)
\frac{d}{d \alpha} \bigg[Y(-\alpha,\gamma)\bigg]_{\alpha=\gamma=r}\\
& =
-\omega_E\left(\frac{\sqrt{N}}{2\pi}\right)^{-2r}\frac{\Gamma(1-r)}{\Gamma(1+r)}\zeta(1+2r)
A(-r,r).\label{rtwofirstfam}
\end{align}
Replacing $\eqref{ronefirstfam}$ and $\eqref{rtwofirstfam}$ in Conjecture \ref{ratiosconjecture}, and using $\alpha=\gamma=r$, we get the desired formula
for
$$
\frac{1}{|\FX|} \sum_{E \in \FX} \frac{L'(\frac{1}{2} + r,E)}{L(\frac{1}{2} + r,E)} .
$$
We still have to justify that the error term remains the same under differentiation. We claim that
\begin{align} \nonumber
& \frac{1}{|\FX|} \Bigg\{\sum_{E \in \FX} \bigg[\frac{L'(\frac{1}{2} + r,E)}{L(\frac{1}{2} + r,E)}-\Big[-\frac{\zeta'}{\zeta}(1+2r) + A_{\alpha}(r,r)\\
 & - \omega_E\left(\frac{\sqrt{N_E}}{2\pi}\right)^{-2r}\frac{\Gamma(1-r)}{\Gamma(1+r)} \zeta(1+2r)A(-r,r) \Big]\bigg]\Bigg\} = O(X^{-1/2+\varepsilon}). \label{ralpha}
\end{align}

Let the left hand side of $\eqref{ralpha}$ be denoted by
$R(\alpha)$. Let $\alpha_0 \in \CC$ such that ${\rm
Re}(\alpha_0)>0.$ Assume $R$ is analytic in a neighborhood of
$\alpha_0$ and let $C$ be a circle of radius $r_0 \approx 1$ around
$\alpha_0$. Then by Cauchy's integral formula we have
\begin{align}
|R'(\alpha_0)|&=\left|\frac{1}{2\pi
i}\oint_C\frac{R(\alpha)}{(\alpha-\alpha_0)^2}d\alpha\right|\leq
\max_{\alpha \in C}\left|\frac{R(\alpha)}{2 \pi i
(\alpha-\alpha_0)^2}\right|(2 \pi r_0) \nonumber \\
&\leq \max_{\alpha \in
C}\left|\frac{1}{(\alpha-\alpha_0)^2}\right|O(X^{-1/2+\varepsilon})=O(X^{-1/2+\varepsilon})
\end{align}
from our assumption in Conjecture \ref{ratiosconjecture}. This completes the proof. \end{proof}

We now use Theorem \ref{ratiostheorem} to rewrite the one-level density $D(\mathcal{F}(X), \phi)$ for the family of all elliptic curves.
With the change of variable $s \rightarrow 1-s$ in \reff{eq:Cauchy} (noting that $\phi$
is even), we get that
\begin{eqnarray} \nonumber
&&\frac{1}{|\FX|} \sum_{E \in \FX} \frac{1}{2\pi i} \int_{(1-c)}
\frac{L'(s,E)}{L(s,E)} \phi(-i(s-\tfrac{1}{2}))ds \\ \label{replaceoneminus}
=&& \frac{1}{|\FX|} \sum_{E \in \FX}
\frac{1}{2\pi i} \int_{(c)} \frac{L'(1-s,E)}{L(1-s,E)}
\phi(-i(s-\tfrac{1}{2}))ds.
\end{eqnarray}
The functional equation
\begin{equation}
L(s,E) = \omega_E X(s,E) L(1-s,E)
\end{equation}
where $X(s,E)$ is defined in \reff{definitionofXE} gives
\begin{equation} \label{logfunctional}
\frac{L'(s,E)}{L(s,E)} = \frac{X'(s,E)}{X(s,E)} -
\frac{L'(1-s,E)}{L(1-s,E)}.
\end{equation}
Using \reff{logfunctional} in \reff{replaceoneminus} gives that the
second integral of \reff{eq:Cauchy} can be rewritten as
\begin{equation*}
\frac{1}{|\FX|} \sum_{E \in \FX} \frac{1}{2\pi i} \int_{(c)}
\left[\frac{L'(s,E)}{L(s,E)} - \frac{X'(s,E)}{X(s,E)} \right]
\phi(-i(s-\tfrac{1}{2}))ds,
\end{equation*}
and then
\begin{align*}
& D({\mathcal{F}}; \phi, X) \\
=& \frac{1}{|\FX|} \sum_{E \in \FX}\frac{1}{2\pi i}\left(\int_{(c)} -
\int_{(1-c)}\right)\frac{L'(s,E)}{L(s,E)} \phi(-i(s-\tfrac{1}{2}))
ds\\
=&\frac{1}{|\FX|} \sum_{E \in \FX} \frac{1}{2\pi
i}\int_{(c)} \bigg[2 \frac{L'(s,E)}{L(s,E)} -\frac{X'(s,E)}{X(s,E)}\bigg] \phi(-i(s-\tfrac{1}{2})) ds \\
=& \frac{1}{|\FX|} \sum_{E \in \FX} \frac{1}{2\pi i}\int_{(c-1/2)} \bigg[2 \frac{L'(1/2 +
r,E)}{L(1/2 + r,E)} - \frac{X'(1/2+r,E)}{X(1/2+r,E)}\bigg] \phi(-ir)dr
\end{align*}
with the change of variable $s = 1/2 + r$. We bring the summation
inside the integral and substitute
\begin{equation*}
\frac{1}{|\FX|}\sum_{E \in \FX}
\frac{L'(\frac{1}{2}+r,E)}{L(\frac{1}{2}+r,E)}
\end{equation*}
with the expression of Theorem \ref{ratiostheorem}, and we pull the summation back
outside of the integral. This gives
\begin{align*}
&D({\mathcal{F}}; \phi, X)\\
=&\frac{1}{|\FX|} \sum_{E\in \FX}
\frac{1}{2\pi i} \int_{(c-1/2)}  \bigg[-2 \frac{\zeta'}{\zeta}(1+2r)
+ 2 A_{\alpha}(r,r) \\
-&2\omega_E\left(\frac{\sqrt{N_E}}{2\pi}\right)^{-2r}\frac{\Gamma(1-r)}{\Gamma(1+r)}
\zeta(1+2r)A(-r,r)  - \frac{X'(1/2+r,E)}{X(1/2+r,E)} \bigg] \, \phi(-ir) \, dr \\
+&O(X^{-1/2+\varepsilon}).
\end{align*}

We now move the integral from $\mbox{Re}(r)=c-1/2=c'$ to
$\mbox{Re}(r)=0$ by integrating over the rectangle $R$ from $c'-iT$
to $c'+iT$ to $iT$ to $-iT$ and back to $c'-iT$, and letting $T
\rightarrow \infty$. The two horizontal integrals tend to 0, and we
only have to consider the vertical integrals. We have to distinguish
2 cases, as the integrand
\begin{eqnarray*}
F(r) &=& - \frac{X'(1/2+r,E)}{X(1/2+r,E)} - 2 \frac{\zeta'}{\zeta}(1+2r) + 2 A_{\alpha}(r,r) \\
&& -2
\omega_E\left(\frac{\sqrt{N_E}}{2\pi}\right)^{-2r}\frac{\Gamma(1-r)}{\Gamma(1+r)}
\zeta(1+2r)A(-r,r)
\end{eqnarray*}
has a pole at $r=0$ with residue 2 on the boundary of the rectangle $R$ when
$\omega_E=-1$. We have that the function $$F(r) - \frac{1-\omega_E}{r}$$ is analytic inside and on the contour $R$. Hence from
Cauchy's Theorem we have that
\begin{align*} \nonumber
D({\mathcal{F}}; \phi, X) &=  \frac{1}{|\FX|} \sum_{E\in \FX} \frac{1}{2\pi}
\int_{-\infty}^\infty \bigg[- 2 \frac{\zeta'}{\zeta}(1+2it) + 2
A_{\alpha}(it,it) \\ \nonumber &-2
\omega_E\left(\frac{\sqrt{N_E}}{2\pi}\right)^{-2it}\frac{\Gamma(1-it)}{\Gamma(1+it)}
\zeta(1+2it)A(-it,it) \\ \nonumber &+
2\log\left(\frac{\sqrt{N_E}}{2\pi}\right)+\frac{\Gamma'}{\Gamma}(1-it)
+ \frac{\Gamma'}{\Gamma}(1+it)  - \frac{1-\omega_E}{it}\bigg] \,
\phi(t) \, dt \\
&+ \frac{1}{|\FX|} \sum_{{E\in \FX}\atop{\omega_E = -1}} \frac{1}{2 \pi i} \int_{(c')}   \frac{2\phi(-ir)}{r}  \,dr + O(X^{-1/2+\varepsilon}),
\end{align*}
where we used the change of variable $r=it$ in the first integral, and
\begin{equation*}
\frac{X'(1/2 + it,E)}{X(1/2 + it,E)} =
-2\log\left(\frac{\sqrt{N_E}}{2\pi}\right)-\frac{\Gamma'}{\Gamma}(1-it)
- \frac{\Gamma'}{\Gamma}(1+it).
\end{equation*}

If $\omega_E=1$, then ${1-\omega_E}=0$, and the second sum is zero. If $ \omega_E=-1$, then by Cauchy's Theorem
\begin{equation*}
2 \phi(0)= \frac{1}{2 \pi i} \int_{(c')}   \frac{2 \phi(-ir)}{r} dr - \frac{1}{2 \pi i} \int_{(-c')}   \frac{2 \phi(-ir)}{r} dr =
\frac{2}{2 \pi i} \int_{(c')}  \frac{2 \phi(-ir)}{r} dr ,
\end{equation*}
which gives the following theorem.

\begin{theorem}\label{oneleveldensityresult}
Assuming the Ratios Conjecture \ref{ratiosconjecture}, the one-level density of the family $\FX$ of all elliptic curves is given by
\begin{align*} \nonumber
D({\mathcal{F}}; \phi, X) & = \frac{1}{|\FX|}\frac{1}{2\pi} \int_{-\infty}^\infty
\phi(t) \sum_{E\in \FX}\bigg[ 2 \log\bigg(\frac{\sqrt{N_E}}{2\pi}
\bigg) + \frac{\Gamma'}{\Gamma}(1+it) \\ \nonumber & +
\frac{\Gamma'}{\Gamma}(1-it) + 2 \bigg(-\frac{\zeta'}{\zeta}(1+2it)
+ A_{\alpha}(it,it) \\  \nonumber &
-\omega_E\left(\frac{\sqrt{N_E}}{2\pi}\right)^{-2it}\frac{\Gamma(1-it)}{\Gamma(1+it)}
\zeta(1+2it)A(-it,it)\bigg) - \frac{1-\omega_E}{it} \bigg]dt \\
&+  \frac{\phi(0)}{|\FX|} \sum_{{E\in
\FX}\atop{\omega_E=-1}} 1 + O(X^{-1/2+\varepsilon}).
\end{align*}
\end{theorem}

\kommentar { {\bf TO DO:} Now, we have to make the change of
variable $\tau=\frac{t \log{X}}{2 \pi}$ and analyze the first few
terms of the power series. I think we have to separate $\omega_E=1$
and $\omega_E=-1$. When $\omega_E=1$, we get $1 + \frac{\sin{2 \pi
\tau}}{2 \pi \tau}$ (corresponding to $SO(\mbox{even})$), and when
$\omega_E=-1$, we get $1 - \frac{\sin{2 \pi \tau}}{2 \pi \tau}$
(corresponding to $SO(\mbox{odd})$). The Dirac function that shows
in $SO(\mbox{odd})$ comes from
$$
\phi(0) = \int_{-\infty}^{\infty} \phi(t) \delta_0(t) \;dt.$$ }

In \cite[Theorem 3.1]{Young2006}, Young showed that the underlying symmetry of the
family $\FX$ is orthogonal for test functions $\phi$
with $\widehat{\phi}\subset (-\tfrac{7}{9},\tfrac{7}{9})$. Thus we
expect
\begin{equation*}
\lim_{X \rightarrow \infty} D({\mathcal{F}}; \phi, X) = \int_{-\infty}^\infty
\phi(t) \mathcal{W}(\mathcal{F})(t) dt
\end{equation*}
where
\begin{equation*}
\mathcal{W}(\mathcal{F})(t) = 1 + \tfrac{1}{2} \delta_0(t)
\end{equation*}
with $\delta_0(t)$ denoting the Dirac distribution, and
$\mathcal{W}(\mathcal{F})(t)$ is the scaled one-level density of
the orthogonal group. In the following, we show that our result in
Theorem \ref{oneleveldensityresult} is also consistent with this
expectation. We also write lower order terms for the scaled one-level density
by expanding the expression in Theorem \ref{oneleveldensityresult} in a Taylor series.

We first make the standard change of variables
\begin{equation} \label{changeofV}
\tau = \frac{t L}{\pi} \mbox{~with~} L =
\log\left(\frac{\sqrt{X}}{2\pi e}\right).
\end{equation}
Note that $L$ is chosen so that for $N_E\approx X$, the sequence $\gamma_E$ of low-lying zeros arising from (say) $\gamma_E\leq1$ has essentially constant mean spacing one. (Recall that
by a Riemann-von Mangold type theorem as in for example \cite[Thm.\ 5.8]{I-K-book} $L(s,E)$ has approximately $\log (N_E /(2\pi e)^2)/2\pi$ zeros in the region $0<\Im(\rho_E) \leq 1$). We defined the normalized
test function $\psi$ by
\begin{equation} \label{defofpsi} \phi(t) = \psi\left( \frac{t L}{\pi} \right).\end{equation}

We know from the work of \cite[Lemma 5.1]{Young2006} that the {\it conductor condition}
holds for the family $\FX$, and we can write
\begin{equation} \label{conductor-fromMY}
\frac{1}{|\FX|} \sum_{E \in \FX} \log{ \left( \frac{\sqrt{N_E}}{2 \pi} \right)} = \log{ \left( \frac{\sqrt{X}}{2 \pi} \right)} \left\{ 1 + O \left( \frac{1}{\log{X}} \right) \right\}.
\end{equation}
We make the standard
assumption
that half of the family has even and the other half odd rank (the Parity Conjecture). Even
rank corresponds to an even functional equation ($\omega_E = 1$) and
odd rank corresponds to an odd functional equation ($\omega_E=-1$), and that this is compatible with \eqref{conductor-fromMY}. Then, using the
change of variables \eqref{changeofV},
we rewrite the statement of Theorem \ref{oneleveldensityresult} as
\begin{eqnarray*}
\frac{1}{|\FX|} \sum_{E \in \FX} \sum_{\gamma_E} \psi \left( \frac{\gamma_E L}{\pi} \right) &\sim&
\int_{-\infty}^\infty \psi (\tau)
 g(\tau) \;d\tau + \frac{\phi(0)}{2} \\
 &=& \int_{-\infty}^\infty \psi (\tau)
 h(\tau) \;d\tau \\
\end{eqnarray*}
where
\begin{eqnarray*} \nonumber
g(\tau) &=& h(\tau) - \frac{1}{2} \delta_0(\tau) \\ &=& \frac{1}{2L} \bigg[ 2
\log\left(\frac{\sqrt{X}}{2\pi}\right) +
\frac{\Gamma'}{\Gamma}\bigg(1 + \frac{  \pi i\tau}{L}\bigg) +
\frac{\Gamma'}{\Gamma}\left(1 - \frac{ \pi i\tau}{L}\right) \\
&&
 -2 \frac{\zeta'}{\zeta}\bigg(1 + \frac{2 \pi i\tau}{L}\bigg)
 + 2 A_{\alpha}\bigg(\frac{ \pi i\tau }{L}, \frac{
\pi i\tau }{L}\bigg)  - \frac{L}{ \pi i \tau} \bigg].
\end{eqnarray*}
Making use of the following Laurent series
\begin{align*}
\frac{\zeta'}{\zeta}\left(1 + s\right) &= \frac{-1}{s}+\gamma_0 -(\gamma_0^2+\gamma_1)s+\cdots,\\
\zeta(1+s) &= \frac{1}{s}+\gamma_0-\gamma_1s+\cdots,\\
\frac{\Gamma'}{\Gamma}\left(1+s\right)&=-\gamma_0+\frac{\pi^2}{6}s-\zeta(3)s^2+\cdots,\\
A_\alpha(s,s)&=A_\alpha(0,0)+(A_{\alpha\alpha}(0,0)+A_{\alpha\gamma}(0,0))s+\cdots,
\end{align*}
where  $A_{\alpha\alpha}(r,r) := \frac{d}{d
\alpha}A_\alpha(\alpha,\gamma)\bigg|_{\alpha=\gamma=r},A_{\alpha\gamma}(r,r) := \frac{d}{d
\gamma}A_\alpha(\alpha,\gamma)\bigg|_{\alpha=\gamma=r} $ and the $\gamma_n$ are the Stieltjes constants,
the Taylor expansion of $h(\tau)$  in $L^{-1}$ is
\begin{align*}
h(\tau) =&\frac{1}{2L}\Bigg[2L-2\gamma_0-2\left(\frac{-L}{2\pi i \tau}+\gamma_0-(\gamma_0^2+\gamma_1)\frac{2\pi i \tau}{L}\right) \\
+&2\left(A_\alpha(0,0)+(A_{\alpha\alpha}(0,0)+A_{\alpha\gamma}(0,0))\frac{\pi i \tau}{L}\right)-\frac{L}{\pi i \tau}+O\left(L^{-2}\right)\Bigg] + \frac{1}{2} \delta_0(\tau) \\
=&1+ \frac{1}{2} \delta_0(\tau) + \frac{A_\alpha(0,0)-2\gamma_0}{L}
+\frac{ (A_{\alpha\alpha}(0,0)+A_{\alpha\gamma}(0,0)+2(\gamma_0^2+\gamma_1))\pi i \tau}{L^{2}}\\
+& O\left(\frac{1}{L^{3}}\right).
\end{align*}
Then, the leading terms for the one-level scaling density associated to the
families of all elliptic curves give $\mathcal{W}(\tau) = 1+ \frac{1}{2} \delta_0(\tau)$
which corresponds to the density $\mathcal{W}(O)(\tau)$ associated with the orthogonal
group $O$ as predicted by the conjectures of Katz and Sarnak. We also get
lower order terms for the
one-level scaling density  which
are particular to this family, and
could be used to refine experimental statistics for small conductor.

\section{A one-parameter family of elliptic curves} \label{oneparameter}

We now consider another family of elliptic curves, a one-parameter family where the sign of the
functional equation is always negative. This family was first studied by Washington \cite{washington}, who proved that
the rank is odd for $t^2 + 3t +3$ assuming the finiteness of the Tate-Shafarevic group. Rizzo \cite{rizzo} then
proved that the rank is odd for all $t$. The one-level density for this family was also studied by Miller \cite{miller-thesis, Miller}.

More precisely, let
\begin{equation} \label{oneparafam}
E_t: y^2=x^3+tx^2-(t+3)x+1.
\end{equation}
The discriminant of $E_t$ is
\begin{equation*}
\Delta(t)=2^4(t^2+3t+9)^2.
\end{equation*}
Replacing $t$ with $12t+1$ gives $\Delta(12t+1)=2^4(144t^2+60t+13)^2$. As proven in \cite{Miller},
if $144t^2+60t+13$ is square-free then the conductor is
\begin{equation*} \label{tconductor}
C(t)=2^3(t^2+3t+9)^2.
\end{equation*}

In this section we will study the family
\begin{equation} \label{oneparameterfamily}
\Ft =\{ E_t: y^2=x^3+tx^2-(t+3)x+1, \;\; t \in \ZZ\}.
\end{equation}
As usual, we denote by
\begin{equation*}
\Ft(X) =\{ E_t: t\leq X^{\frac{1}{4}}\},
\end{equation*}
the set of curves in $\Ft$ of conductor of size at most $X$. Let
\begin{equation} \label{after3.3}
L(s,E_t)=\sum_{n = 1}^\infty \frac{\lambda_t(n)}{n^s} = \prod_p
\left(1- \frac{\lambda_t(p)}{p^s} + \frac{\psi_t(p)}{p^{2s}} \
\right)^{-1}
\end{equation}
denote the $L$-function attached to $E_t$ where $\psi_t$ is the
principal Dirichlet character modulo the conductor $C(t)$ of $E$,
i.e.,
\begin{equation*}
\psi_t(p) = \begin{cases}
1 &\:{\rm if}\:p\nmid C(t),\\
0 & \:{\rm if}\: p \mid C(t).
\end{cases}
\end{equation*}

For $p\neq 2$, $\lambda_t(p)$ is given by
\begin{equation*} \label{definitionoflambdat}
\lambda_t(p) = -\frac{1}{\sqrt{p}} \sum_{x \modd p}
\left(\frac{x^3+tx^2 - (t+3)x + 1}{p} \right).
\end{equation*}
If $p=2$ then $\eqref{oneparafam}$ has a cusp and $\lambda_t(2^k)=0$ for all positive $k$. We
 recall that $\lambda_t(n)$ are multiplicative, and prime powers are computed by
 the Hecke relations
\begin{equation*} \label{theckerelations}
\lambda_t(p^j) = \bigg\{\begin{array}{ll}
U_j\left(\frac{\lambda_t(p)}{2}\right) & \qquad {\rm if}\:
(p,C(t))=1, \\
\lambda_t^j(p) & \qquad {\rm if}\: (p,C(t))>1,
\end{array}\bigg.
\end{equation*}
where $U_j(x)$ are the Chebyshev polynomials.

As in the previous section, we use the principal and the dual sums of the
approximate functional equation \eqref{AFE}  at $s= 1/2 + \alpha$, ignoring questions of convergence
or error terms. The principal sum is
\begin{equation} \label{tapprox-principal}
\sum_n\frac{\lambda_t(n)}{n^{\frac{1}{2}+\alpha}}
\end{equation}
and since the sign of the functional equation
is always negative, the dual sum is
\begin{equation} \label{tapprox-dual}
- X_t\left(\frac{1}{2}+\alpha\right)\sum_n\frac{\lambda_t(n)}{n^{\frac{1}{2}-\alpha}},
\end{equation}
where
\begin{equation*}
X_t(s)=\frac{\Gamma\left(\frac{3}{2}-s\right)}{\Gamma\left(\frac{1}{2}+s\right)}
\left(\frac{\sqrt{C(t)}}{2\pi}\right)^{1-2s}.
\end{equation*}

Finally, we write
\begin{equation} \label{tinverse}
\frac{1}{L(s,E_t)}=\prod_p\left(1-\frac{\lambda_t(p)}{p^s}+\frac{\psi_t(p)}{p^{2s}}\right)=\sum_{n=1}^\infty \frac{\mu_t(n)}{n^s},
\end{equation}
where $\mu_t$ is multiplicative and given by
\begin{equation} \label{definitionofmut}
\mu_t(p^k) =
\begin{cases}
-\lambda_t(p) & \mbox{~if~} k = 1,\\
\psi_t(p) & \mbox{~if~} k = 2,\\
0 & \mbox{~if~} k > 2.
\end{cases}
\end{equation}

We are now ready to derive the $L$-function ratios conjecture for this family following the same recipe that was used in the first family. We keep in this section all the notation of the previous section, but the objects are now attached to the new family. Using \eqref{tapprox-principal}, \eqref{tapprox-dual} and \eqref{tinverse}, we set
\begin{align}
R_1(\alpha,\gamma)&:=\frac{1}{|\Ft(X)|}\sum_{E_t\in
\Ft(X)}\sum_{m_1,m_2}\frac{\lambda_t(m_1)\mu_t(m_2)}{m_1^{\frac{1}{2}+\alpha}m_2^{\frac{1}{2}+\gamma}},
 \label{rstarone} \\
R_2(\alpha,\gamma)&:=\frac{-1}{|\Ft(X)|}\sum_{E_t\in\Ft(X)} X_t \left(\frac{1}{2}+\alpha\right) \sum_{m_1,m_2}\frac{\lambda_t(m_1)
\mu_t(m_2)}{m_1^{\frac{1}{2}+\alpha}m_2^{\frac{1}{2}+\gamma}}.
\label{rstartwo}
\end{align}
and we approximate
\begin{equation} \label{replacehere} \frac{1}{|\Ft(X)|}\sum_{E_t\in
\Ft(X)} \frac{L(1/2+\alpha, E_t)}{L(1/2+\gamma, E_t)}
\end{equation}  by $R_1(\alpha,\gamma)+R_2(\alpha,\gamma)$.
As before, the main step to obtain the ratios conjecture for this family is to replace each $(m_1,m_2)$-summand in $\eqref{rstarone}$ by its average over the family.

\begin{lemma}  \label{qstarsecondfam} Let $m_1, m_2 \geq 0$ be fixed integers, and let
\begin{equation*}
\Qt(m_1,m_2)  := \frac{1}{m^*}\sum_{t \modd{m^*}}
\lambda_t(m_1)\mu_t(m_2)
\end{equation*}
where $m^*=\prod_{p\mid m_1m_2} p$.
Then
\begin{align*}
 \lim_{X\rightarrow \infty}\frac{1}{|\FX|}\sum_{E \in
\FX}\lambda_t(m_1)\mu_t(m_2) = \Qt(m_1,m_2). \label{qstarsecondfam}
\end{align*}
Furthermore, $\Qt(m_1,m_2)$ is multiplicative.
\end{lemma}
\begin{proof} This is completely similar to the proof of Lemma  \ref{average}. \end{proof}

Then replacing each term in $\eqref{rstarone}$ by its average value $\Qt(m_1,m_2)$, and using
Lemma $\ref{qstarsecondfam}$, we are led as before to consider
\begin{equation} \label{Htdirichlet}
\Ht(\alpha,\gamma):=\sum_{m_1,m_2}\frac{\Qt(m_1,m_2)}{m_1^{\frac{1}{2}+\alpha}m_2^{\frac{1}{2}+\gamma}}
= \prod_p\sum_{m_1,m_2}
\frac{\Qt(p^{m_1},p^{m_2})}{p^{m_1(\frac{1}{2}+\alpha)+m_2(\frac{1}{2}+\gamma)}}.
\end{equation}
As in the previous case, we switched notation, and we are now using $m_1, m_2$ for the exponents of the prime powers. By the definition of the Moebius function in $\eqref{definitionofmut}$
only the terms with $m_2=0,1,2$ in $\eqref{Htdirichlet}$
contribute. So we have
\begin{align}
& \sum_{m_1, m_2} \frac{\tQtp}{\denominator} \nonumber \\
=&\sum_{m_1 \geq 0} \frac{\tQtpzero}{\denominatoralpha} + \sum_{m_1
\geq 0}\frac{\tQtpone}{p^{m_1(\frac{1}{2} + \alpha) + \frac{1}{2} +
\gamma}} + \sum_{m_1 \geq 0}\frac{\tQtptwo}{p^{m_1(\frac{1}{2} +
\alpha) + 1+ 2\gamma}}\label{HtinQform}
\end{align}
where
\begin{eqnarray*}
\tQtpzero &=& \frac{1}{p} \sum_{t \modd p} \lambda_t(p^{m_1}), \label{Qtzero} \\
\tQtpone &=& -\frac{1}{p} \sum_{t \modd p} \lambda_t(p^{m_1}) \lambda_t(p), \label{Qtone}\\
\tQtptwo &=& \frac{1}{p} \sum_{\substack{t \modd p \\ p\nmid C(t)}}
\lambda_t(p^{m_1}). \label{Qttwo}
\end{eqnarray*}

Let $\chi_4(n)$ denote the non-principal character modulo 4.

\begin{lemma} \label{qstarone}
For $p>2$ we have that
\begin{equation}
\Qt(p,1)=-\Qt(1,p)=-\left(\frac{1+\chi_4(p)}{\sqrt{p}}\right).
\end{equation}
\end{lemma}

\begin{proof} We have
\begin{align*}
\Qt(p,1) &= \frac{1}{p} \sum_{t \modd{p}} \lambda_t(p) = -\Qt(1,p)
\noindent \\
&=\frac{1}{p} \sum_{t \modd{p}}
\left(\frac{-1}{\sqrt{p}}\sum_{x\modd{p}}\left(\frac{(x^2-x)t+(x^3-3x+1)}{p}\right)\right)
\noindent \\
&=\frac{-1}{p^{\frac{3}{2}}}\sum_{x
\modd{p}}\left(\sum_{t\modd{p}}\left(\frac{(x^2-x)t+(x^3-3x+1)}{p}\right)\right).
\end{align*}
We have that
\begin{equation*}
\sum_{t \modd{p}}\left(\frac{(x^2-x)t+(x^3-3x+1)}{p}\right)=\left\{
\begin{array}{ll} p\left(\frac{x^3-3x+1}{p}\right) & {\rm
if}\: x^2 \equiv x \modd{p}, \\
0 & {\rm otherwise}.
\end{array} \right.
\end{equation*}
Thus
\begin{equation*}
\Qt(p,1)=\frac{-1}{\sqrt{p}}\left(\left(\frac{1}{p}\right)+\left(\frac{-1}{p}\right)\right)
=\frac{-\left(1+\chi_4(p)\right)}{\sqrt{p}}.
\end{equation*}
\end{proof}

\begin{lemma} \label{qstartwo}
We have for $p>2$ that
\begin{equation} \label{lemmatwotwo}
\Qt(p,p)= -1+O\left(\frac{1}{p}\right).
\end{equation}
\end{lemma}

\begin{proof}
We have that
\begin{align*}
&\Qt(p,p)= \frac{-1}{p} \sum_{t\modd{p}}
\left(\frac{-1}{\sqrt{p}}\sum_{x
\modd{p}}\left(\frac{(x^2-x)t+(x^3-3x+1)}{p}\right)\right)^2
\nonumber \\
&=\frac{-1}{p^2}\sum_{t,x,y\modd{p}}\left(\frac{(x^2-x)t+(x^3-3x+1)}{p}\right)\left(\frac{(y^2-y)t+(y^3-3y+1)}{p}\right).
\end{align*}

Let
\begin{align*}
a:&=(x^2-x)(y^2-y),\\
b:&=(x^2-x)(y^3-3y+1)+(x^3-3x+1)(y^2-y),\\
c:&=(x^3-3x+1)(y^3-3y+1),\\
d:&=(x-y)(xy-x+1)(xy-y+1),
\end{align*}
then
\begin{align*}
b^2-4ac&=[(x^2-x)(y^3-3y+1)+(x^3-3x+1)(y^2-y)]^2\\
&-4(x^2-x)(y^2-y)(x^3-3x+1)(y^3-3y+1)\\
&=[(x^2-x)(y^3-3y+1)-(x^3-3x+1)(y^2-y)]^2\\
&=d^2.
\end{align*}
From \cite[Exercise 1.1.9]{Stepanov} we have for $a\not \equiv 0 \modd{p}$ that

\begin{equation} \label{sumofquadratic}
\sum_{t \modd{p}}\left(\frac{at^2+bt+c}{p}\right)=\left\{
\begin{array}{ll} \left(\frac{a}{p}\right)(p-1) & {\rm
if}\: b^2-4ac \equiv 0 \modd{p}, \\
-\left(\frac{a}{p}\right) &{\rm if}\: b^2-4ac \not \equiv 0 \modd{p}.\end{array}\right.
\end{equation}
Now let
$$S:=\sum_{\substack{ x,y \modd{p} \\ a\not \equiv 0 \modd{p} \\ d\equiv 0 \modd{p}}}\left(\frac{(x^2-x)(y^2-y)}{p}\right)=\sum_{\substack{ x,y \modd{p} \\ d\equiv 0 \modd{p}}}\left(\frac{(x^2-x)(y^2-y)}{p}\right).$$
If $a\equiv 0 \modd{p}$ then there are four cases to consider, $(x,y)=(0,0),(0,1),(1,0)$ or $(1,1)$. In all 4 cases $b\equiv 0 \modd{p}$. Then we have that
\begin{align}
\Qt(p,p) &=\frac{-S}{p}+\frac{1}{p^2}\sum_{\substack{x,y \modd{p} \\ a\not \equiv 0 \modd{p} \\ d\not \equiv 0 \modd{p}}}\left(\frac{(x^2-x)(y^2-y)}{p}\right)\nonumber \\
&-\frac{1}{p}\sum_{\substack{ x,y \modd{p} \\ a\equiv b\equiv 0 \modd{p}}}\left(\frac{(x^3-3x+1)(y^3-3y+1)}{p}\right)\nonumber\\
&=\frac{-1}{p}S+\frac{1}{p^2}\sum_{\substack{x,y \modd{p} \\ a\not \equiv 0 \modd{p}}}\left(\frac{(x^2-x)(y^2-y)}{p}\right)\nonumber\\
&-\frac{1}{p}\sum_{\substack{ x,y \modd{p} \\ a\equiv b\equiv 0 \modd{p}}}\left(\frac{(x^3-3x+1)(y^3-3y+1)}{p}\right).\label{annoyinglemma}
\end{align}
First we consider the third sum in $\eqref{annoyinglemma}$ and compute
\begin{align}
&\sum_{\substack{ x,y \modd{p} \\ a\equiv b\equiv 0 \modd{p}}}\left(\frac{(x^3-3x+1)(y^3-3y+1)}{p}\right)\nonumber \\
&=\sum_{x,y \equiv 0,1 \modd{p}}\left(\frac{(x^3-3x+1)(y^3-3y+1)}{p}\right)=2\left[1+\left(\frac{-1}{p}\right)\right] \label{gaahone}.
\end{align}

Next we consider the second sum in $\eqref{annoyinglemma}$ and compute
\begin{align}
\sum_{\substack{x,y \modd{p} \\ a\not \equiv 0 \modd{p}}}\left(\frac{(x^2-x)(y^2-y)}{p}\right)&=\sum_{x,y\modd{p}}\left(\frac{(x^2-x)(y^2-y)}{p}\right) \nonumber \\ &=\left[
\sum_{x\modd{p}}\left(\frac{x^2-x}{p}\right)\right]^2=\left[-\left(\frac{1}{p}\right)\right]^2=1.\label{gaaahtwo}
\end{align}

If $d=(x-y)(xy-x+1)(xy-y+1) \equiv 0 \modd{p}$ then either $x\equiv y\modd{p}$, $xy-x+1\equiv 0\modd{p}$ or $xy-y+1\equiv 0 \modd{p}$. All three equations are satisfied when $x^2-x+1\equiv 0 \modd{p}$, which has at most two solutions $\modd{p}$. This gives
\begin{align}
S=&\sum_{x\modd{p}}\left(\frac{(x^2-x)^2}{p}\right)\nonumber\\
+&\left[\sum_{\substack{xy-x+1 \equiv 0 \modd{p} \\ x\equiv \not \equiv 0 \mod{p}}} + \sum_{\substack{xy-y+1 \equiv 0 \modd{p} \\ x\equiv \not \equiv 0 \mod{p} \\ xy-x+1 \not \equiv 0 \modd{p}}} \right]\left(\frac{(x^2-x)(y^2-y)}{p}\right)\label{sumofthrees}
\end{align}
For the the second two sums in $\eqref{sumofthrees}$ for each $x\modd{p}$ there is at most one $y$ satisfying the equation $x\equiv y \modd{p}$, $xy-x+1\equiv 0 \modd{p}$ or $xy-y+1\equiv 0 \modd{p}$ which gives $S=p+O(1).$

Then substituting $\eqref{gaahone}$, $\eqref{gaaahtwo}$ in $\eqref{annoyinglemma}$ gives Lemma \ref{qstartwo}.
\end{proof}

We remark that for any one-parameter family of elliptic curves over $\QQ(t)$ with non-constant $j(E_t)$, we have the estimate $\Qt(p,p)= -1+O\left( p^{-1/2} \right)$ due to Michel \cite{Michel}, which is used for example in \cite[Section 6.1.3]{Miller} for the same
family.

\begin{lemma} \label{lastlemma} Let $p>2$. Then,
\begin{eqnarray*}
\Qt(1,p^2) &=& 1 + O\left(\frac{1}{p}\right), \\
\Qt(p^2,1) &=& O\left(\frac{1}{p}\right) .
\end{eqnarray*}
\end{lemma}

\begin{proof}
For $p>2$, we compute
\begin{equation*}
\Qt(1,p^2)= \frac{1}{p} \sum_{\substack{t \modd p \\ p\nmid C(t)}}
1= 1- \frac{1}{p}\sum_{\substack{t \modd p \\ p \mid C(t)}}1.
\end{equation*}
Since there are at most 2 solutions to the congruence $C(t) \equiv 0 \modd{p}$, this gives
\begin{equation*}
\Qt(1,p^2) = 1 + O\left(\frac{1}{p}\right).
\end{equation*}
We also have that
\begin{equation*}
\Qt(p^2,1) =\frac{1}{p}\sum_{t\modd{p}}\lambda_t(p^2),
\end{equation*}
and from $\eqref{heckerelations}$
\begin{equation*}
\lambda_t(p^2) = \bigg\{\begin{array}{ll}
U_2\left(\frac{\lambda_t(p)}{2}\right) = \lambda_t^2(p)-1 & \qquad
{\rm if}\:(p,C(t))=1, \\
\lambda_t^2(p) & \qquad {\rm if}\: (p,C(t))>1.
\end{array}\bigg.
\end{equation*}
Hence we compute
\begin{align*}
\Qt(p^2,1)&= \frac{1}{p}\sum_{\substack{ t \modd{p} \\ p \nmid
C(t)}} \left(\lambda_t^2(p)-1\right)+
\frac{1}{p}\sum_{\substack{ t \modd{p} \\ p \mid C(t)}} \lambda_t^2(p) \\
&= -\Qt(p,p)-\Qt(1,p^2) \\
&=
-\left(-1+O\left(\frac{1}{p}\right)\right)-\left(1+O\left(\frac{1}{p}\right)\right)
= O\left(\frac{1}{p}\right).
\end{align*}
\end{proof}

Finally, if $p=2$, we have
$\Qt(2^{m_1},2^{m_2})=0$ if $(m_1,m_2) \neq (0,0)$. Now we are now ready to prove the following result.
\begin{theorem} \label{factorHdash}
Let $\Ht$ be given by $\eqref{Htdirichlet}$. Then $\Ht$ has the form
\begin{equation} \label{HAzetafactors}
\Ht(\alpha, \gamma) =
\frac{\zeta(1+2\gamma)\zeta(1+\gamma)}{\zeta(1+\alpha+\gamma)\zeta(1+\alpha)}\At(\alpha,\gamma)
\end{equation}
where $\At(\alpha,\gamma)$ is holomorphic and non-zero for $\RRe(\alpha), \RRe(\gamma) >  -1/4$.
\end{theorem}

\begin{proof} We have that $\Qt(1,1) = 1$, and from $\eqref{Htdirichlet}$ and \eqref{HtinQform}, we have
\begin{align*}
&\Ht(\alpha, \gamma) = \prod_p\left(1+\sum_{m_1 \geq 1}
\frac{\tQtpzero}{\denominatoralpha} + \sum_{m_1 \geq
0}\frac{\tQtpone}{p^{m_1(\frac{1}{2} + \alpha) +\frac{1}{2} +
\gamma}} + \sum_{m_1 \geq 0}\frac{\tQtptwo}{p^{m_1(\frac{1}{2} +
\alpha) + 1+ 2\gamma}}\right).\\
&=\prod_p\left(1
+\frac{\Qt(p,1)}{p^{\frac{1}{2}+\alpha}}+\frac{\Qt(p^2,1)}{p^{1+2\alpha}}+\frac{\Qt(1,p)}{p^{\frac{1}{2}+\gamma}}
+\frac{\Qt(p,p)}{p^{1+\alpha+\gamma}}+\frac{\Qt(1,p^2)}{p^{1+2\gamma}}+T_p{(\alpha,\gamma)}\right),
\end{align*}
where $$T_p{(\alpha,\gamma)}:=\sum_{\substack{m_1\geq 3, m_2=0 \\ m_1\geq 2, m_2=1 \\ m_1\geq 1, m_2=2}}\frac{\Qt(p^{m_1},p^{m_2})}{p^{m_1(\frac{1}{2}+\alpha)+m_2(\frac{1}{2}+\gamma)}}.$$

Using the
formulas from Lemma \ref{qstarone}, \ref{qstartwo} and \ref{lastlemma} in $H(\alpha, \gamma)$, we obtain
\begin{align*}
\Ht(\alpha, \gamma)
= \prod_{p>2}\Bigg(1&+\left(1+\chi_4(p)\right)\left(\frac{1}{p^{1+\gamma}}-\frac{1}{p^{1+\alpha}}
\right)-\frac{1}{p^{1+\alpha+\gamma}}+\frac{1}{p^{1+2\gamma}}\\
&+T_p{(\alpha,\gamma)}+O\left(\frac{1}{p^{2+2\alpha}}+\frac{1}{p^{2+2\gamma}}+\frac{1}{p^{2+\alpha+\gamma}}\right)\Bigg).
\end{align*}

Then,
\begin{equation}
\Ht(\alpha, \gamma)= A(\alpha, \gamma) \prod_{p > 2}
\left(1+\left(1+\chi_4(p)\right)\left(\frac{1}{p^{1+\gamma}}-\frac{1}{p^{1+\alpha}}
\right)-\frac{1}{p^{1+\alpha+\gamma}}
+\frac{1}{p^{1+2\gamma}}\right)\label{etavgeuler}
\end{equation}
where $A(\alpha, \gamma)$ is analytic for $\RRe(\alpha), \RRe(\gamma) >  -1/4$.

Let $\zeta_K(s)$ be the Dedekind zeta function of $K={\mathbb{Q}}(i)$, i.e.,
\begin{align*}
\zeta_K(s)&=\left(1-\frac{1}{2^s}\right)^{-1}\prod_{p\equiv 1
\modd{4}}\left(1-\frac{1}{p^s}\right)^{-2}\prod_{p\equiv3
\modd{4}}\left(1-\frac{1}{p^{2s}}\right)^{-1} .
\end{align*}
Since $\zeta_K(s)$ has a simple pole at $s=1$, the factors
$$
\frac{1+\chi_4(p)}{p^{1+\gamma}} \quad \mbox{and} \quad -\frac{1+\chi_4(p)}{p^{1+\alpha}}
$$
contribute respectively a pole and a zero to $H(\alpha, \gamma)$.
So now we factor out the zeta factors. We can then write (renaming $A(\alpha, \gamma)$)
\begin{equation} \label{htzetas}
\Ht(\alpha, \gamma) =
\frac{\zeta(1+2\gamma)\zeta(1+\gamma)}{\zeta(1+\alpha+\gamma)\zeta(1+\alpha)}\At(\alpha,\gamma),
\end{equation}
where $\At(\alpha,\gamma)$ is analytic for $\RRe(\alpha), \RRe(\gamma) >  -1/4$.
\end{proof}
Finally, we define
\begin{equation} \label{ytfam}
\Yt(\alpha,\gamma):=\frac{\zeta(1+2\gamma)\zeta(1+\gamma)}{\zeta(1+\alpha+\gamma)\zeta(1+\alpha)}.
\end{equation}

Working as in our previous family, we replace $R_1(\alpha,\gamma)$ in $\eqref{rstarone}$
with $Y(\alpha,\gamma) A(\gamma, \alpha)$, and $R_2(\alpha,\gamma)$ in $\eqref{rstartwo}$
with
\begin{equation*}
-\left(\frac{\sqrt{C(t)}}{2\pi}\right)^{-2\alpha}\frac{\Gamma\left(1-\alpha\right)}{\Gamma\left(1
+\alpha\right)}\Yt(-\alpha,\gamma)\At(-\alpha,\gamma).
\end{equation*}

Then, replacing each summand in \eqref{replacehere}, we obtain the following conjecture.

\begin{conjecture}[Ratios Conjecture] \label{tratiosconjecture}
Let $\varepsilon >0$.
Let $\alpha, \gamma \in \CC$ such that $\RRe(\alpha) > -1/4$, $\RRe(\gamma) \gg \frac{1}{\log{X}}$
and $\IIm(\alpha), \IIm(\gamma) \ll_\varepsilon X^{1-\varepsilon}$. Then
\begin{align*} \nonumber
& \frac{1}{|\Ft(X)|} \sum_{E_t \in \Ft(X)} \frac{L(\frac{1}{2} +
\alpha,E_t)}{L(\frac{1}{2} + \gamma,E_t)}\\ \nonumber
& = \frac{1}{|\Ft(X)|} \sum_{E_t \in \Ft(X)} \left[\Yt(\alpha,\gamma)\At(\alpha,\gamma) - \left(\frac{\sqrt{C(t)}}{2\pi}\right)^{-2\alpha}\frac{\Gamma(1-\alpha)}{\Gamma(1+\alpha)} \Yt(-\alpha,\gamma)\At(-\alpha,\gamma) \right] \\
& + O(X^{-1/2+\varepsilon})
\end{align*}
where $\Yt(\alpha,\gamma)$ is defined in $\eqref{ytfam}$ and
$\At(\alpha,\gamma)$ in $\eqref{htzetas}$.
\end{conjecture}

As in the previous family, we now use $\alpha=\gamma=r$. We first show that $H(r,r)= A(r,r)=1$. For the previous family, we had a closed form for $H(\alpha,\gamma)$ that
we used to show that $H(r,r)=1$, but this is in fact true for any family
by the Hecke relations as we show in the next lemma.

\begin{lemma} \label{generalEF}
We have
\begin{equation}
H(r,r) =A(r,r) = 1.
\end{equation}
\end{lemma}

\begin{proof}
Using $\eqref{Htdirichlet}$ and the definition of $\Qt(m_1, m_2)$ from Lemma \ref{qstarsecondfam}, we have that
\begin{equation*}
H(r,r) =\prod_p \frac{1}{p} \sum_{t \modd{p}}
\sum_{m_1} \frac{\lambda_t(p^{m_1})}{p^{m_1(\frac{1}{2} + r)}}
\sum_{m_2} \frac{\mu_t(p^{m_2})}{p^{m_2(\frac{1}{2} + r)}}.
\end{equation*}
From the Hecke relations \eqref{after3.3} and \eqref{tinverse}, the $m_1$-sum is
$$\left(1-\frac{\lambda_t(p)}{p^s}+\frac{\psi_t(p)}{p^{2s}}\right)^{-1}
$$
and the $m_2$-sum is
$$\left(1-\frac{\lambda_t(p)}{p^s}+\frac{\psi_t(p)}{p^{2s}}\right).
$$
This proves that $H(r,r)=1$, and $A(r,r)=1$ by \eqref{htzetas}.
\end{proof}

To get the one-level density for the family $\Ft$, we have to differentiate the result of Conjecture
\ref{tratiosconjecture} with respect to $\alpha$ and use $\eqref{eq:Cauchy}$. We define
\begin{equation}\label{atprime}
\At_\alpha(r,r) :=
\frac{\partial}{\partial \alpha}\At(\alpha,\gamma)\bigg|_{\alpha=\gamma=r},
\end{equation}
and we obtain the following theorem.

\begin{theorem} \label{tratiostheoremone} Let $\varepsilon > 0$, and $r \in \CC$.
Assuming the Ratios Conjecture \ref{tratiosconjecture},  $\RRe(r) \gg \frac{1}{\log{X}}$ and
$\IIm(r) \ll X^{1-\varepsilon}$,
we have
\begin{align*} \nonumber
& \frac{1}{|\Ft(X)|} \sum_{E_t \in \Ft(X)} \frac{L'(\frac{1}{2} +
r,E_t)}{L(\frac{1}{2} + r,E_t)} \\
 =& \frac{1}{|\Ft(X)|} \sum_{E_t \in
\Ft(X)}  \bigg[-\frac{\zeta'}{\zeta}\left(1+2r\right)
\nonumber
-\frac{\zeta'}{\zeta}\left(1+r\right)+\At_\alpha(r,r)
\nonumber \\
+& \left(\frac{\sqrt{C(t)}}{2\pi}\right)^{-2r}\frac{\Gamma(1-
r)}{\Gamma(1+r)}\frac{\zeta(1+2r)\zeta(1+r)}{\zeta(1-r)}\At(-r,r)\bigg]+O(X^{-1/2+\varepsilon})
\end{align*}
where $\At_\alpha(r,r)$ is defined in $\eqref{atprime}$.
\end{theorem}

\begin{proof}
We have that
\begin{equation*}
\Yt_\alpha(r,r)=\frac{\partial}{\partial \alpha}\Yt(\alpha,\gamma)\bigg|_{\alpha=\gamma=r}=-\frac{\zeta'}{\zeta}\left(1+2r\right)-\frac{\zeta'}{\zeta}\left(1+r\right).
\end{equation*}
Differentiating the first term in the sum in Conjecture \ref{tratiosconjecture}
gives
\begin{equation*}
-\frac{\zeta'}{\zeta}\left(1+2r\right)-\frac{\zeta'}{\zeta}\left(1+r\right)+\At_\alpha(r,r).
\end{equation*}
For the second term, we compute that
\begin{equation*}
\Yt_\alpha(-r,r)=\frac{\partial}{\partial \alpha}\Yt(-\alpha,\gamma)\bigg|_{\alpha=\gamma=r}=-\frac{\zeta(1+2r)\zeta(1+r)}{\zeta(1-r)}.\label{yalpha}
\end{equation*}
Since $Y(-r,r)=0$, differentiating the second term in Conjecture
\ref{tratiosconjecture} gives
\begin{equation*}
\frac{\partial}{\partial \alpha}R_2(\alpha,\gamma)\bigg|_{\alpha=\gamma=r}=\left(\frac{\sqrt{C(t)}}{2\pi}\right)^{-2r}\frac{\Gamma(1-
r)\zeta(1+2r)\zeta(1+r)\At(-r,r)}{\Gamma(1+r)\zeta(1-r)}\label{rtwoprime}.
\end{equation*}
\end{proof}

As before, we now use the relation
\begin{equation*}
\frac{L'(s,E_t)}{L(s,E_t)}=\frac{X_t'(s)}{X_t(s)}-\frac{L'(1-s,E_t)}{L(1-s,E_t)}
\end{equation*}
where
\begin{equation} \label{logXtevaluated}
\frac{X_t'(s)}{X_t(s)}\bigg|_{s=\frac{1}{2} + r} =
-2\log\left(\frac{\sqrt{C(t)}}{2\pi}\right)-\frac{\Gamma'}{\Gamma}(1-r)
- \frac{\Gamma'}{\Gamma}(1+r)
\end{equation}
and from $\eqref{logXtevaluated}$, we obtain
\begin{align*}
D(\F; \phi, X) & = \frac{1}{|\Ft(X)|}\sum_{E_t\in
\Ft(X)}\frac{1}{2\pi i} \int_{(c-\frac{1}{2})}
\Bigg[-2\left[\frac{\zeta'}{\zeta}\left(1+2r\right)
+\frac{\zeta'}{\zeta}\left(1+r\right)\right]
\nonumber \\
&+2\At_\alpha(r,r)+2\left(\frac{\sqrt{C(t)}}{2\pi}\right)^{-2r}\frac{\Gamma(1-
r)}{\Gamma(1+r)}\frac{\zeta(1+2r)\zeta(1+r)}{\zeta(1-r)}\At(-r,r)
\nonumber \\
&-\frac{X'\left(\frac{1}{2}+r,E_t\right)}{X\left(\frac{1}{2}+r,E_t\right)}
\Bigg]\phi(-ir)dr + O(X^{-1/2+\varepsilon}).
\end{align*}

As in the previous family we move the integral from
$\mbox{Re}(s)=c-\frac{1}{2}=c'$ to $\mbox{Re}(s)=0$ by integrating
over the rectangle $R$ from $c'-iT$ to $c'+iT$ to $iT$ to $-iT$ and
back to $c'-iT$, and letting $T \rightarrow \infty$. The two
horizontal integrals tend to 0, and we only have to consider the
vertical integrals. From $\eqref{logXtevaluated}$ we have that the
integrand becomes
\begin{align*}
F(r) &= \Bigg[-2\left[\frac{\zeta'}{\zeta}\left(1+2r\right)
+\frac{\zeta'}{\zeta}\left(1+r\right)\right]
\nonumber \\
&+2\At_\alpha(r,r)+2\left(\frac{\sqrt{C(t)}}{2\pi}\right)^{-2r}\frac{\Gamma(1-
r)}{\Gamma(1+r)}\frac{\zeta(1+2r)\zeta(1+r)}{\zeta(1-r)}\At(-r,r)
\nonumber \\
&-\left(-2\log\left(\frac{\sqrt{C(t)}}{2\pi}\right)-\frac{\Gamma'}{\Gamma}(1-r)
- \frac{\Gamma'}{\Gamma}(1+r)\right) \Bigg].
\end{align*}
We make use of the Laurent series
\begin{equation*}
\zeta(1-s)^{-1} =-s-\gamma_0s^2-(\gamma_0^2+\gamma_1)s^3 +\cdots,
\end{equation*}
where the $\gamma_n$ are the Stieltjes constants. This gives
\begin{align*}
F(r)&=-2\left(\frac{-1}{2r}+O(1)\right)-2\left(\frac{-1}{r}+O(1)\right)\\
&+2\left(1+O(r)\right)\left(\frac{1}{2r}+O(1)\right)\left(\frac{1}{r}+O(1)\right)\left(-r+O(r^2)\right)+O(1)\\
&=\frac{2}{r}+O(1).
\end{align*}
So there is a pole at $r=0$ with residue $2$ on the boundary of the rectangle $R$.Then $F(r) - \frac{2}{r}$
is an analytic function inside and on the contour $R$. Setting $r=iu$ as in the previous family, and using the
fact that
$$
\frac{1}{|\Ft(X)|} \sum_{E_t \in \Ft(X)} \frac{1}{2 \pi i} \int_{(c')} \frac{2 \phi(-ir)}{r} dr = 2 \phi(0),
$$
we obtain the following result.

\begin{theorem} \label{tratiostheoremtwo}
Assuming the Ratios Conjecture \ref{tratiosconjecture}, the one-level density of
the family $\Ft$ defined by \eqref{oneparameterfamily} is given by
\begin{align} \nonumber
D(\F; \phi, X) & = \frac{1}{|\Ft(X)|}\frac{1}{2\pi}
\int_{-\infty}^{\infty}\phi(u)\sum_{E_t\in \Ft(X)}
\Bigg[2\log\left(\frac{\sqrt{C(t)}}{2\pi}\right)+\frac{\Gamma'}{\Gamma}\left(1+iu\right)\nonumber
\\
&+\frac{\Gamma'}{\Gamma}\left(1-iu\right)-2\bigg(\frac{\zeta'}{\zeta}\left(1+2iu\right)
+\frac{\zeta'}{\zeta}\left(1+iu\right)\bigg)+2\At_\alpha(iu,iu)\nonumber \\
&+2\left(\frac{\sqrt{C(t)}}{2\pi}\right)^{-2iu}\frac{\Gamma(1-
iu)}{\Gamma(1+iu)}\frac{\zeta(1+2iu)\zeta(1+iu)}{\zeta(1-iu)}\At(-iu,iu)-\frac{2}{iu}\Bigg]du
\nonumber \\
&+ \phi(0)+ O(X^{-1/2+\varepsilon}).\label{tfamoneleveldensity}
\end{align}
\end{theorem}

As in the previous family we make the change of variables,
$$\tau = \frac{uL}{\pi}\quad {\rm with}\quad L=\log\left(\frac{\sqrt{X}}{2\pi}\right),$$
and we define the test function $\psi$ by \eqref{defofpsi}.


\begin{lemma}
Let $\Ft$ be the family defined by defined by \eqref{oneparameterfamily}.
Then,
\begin{equation} \label{secondconductorcondition}
\frac{1}{|\FX|}\sum_{E_t\in \FX}\log\left(\frac{\sqrt{C(t)}}{2\pi}\right)\sim\log\left(\frac{\sqrt{X}}{2\pi}\right).
\end{equation}
\end{lemma}
\begin{proof}
Let $T=X^{\frac{1}{4}}$. Showing $\eqref{secondconductorcondition}$ is equivalent to showing that
$$\frac{1}{T}\sum_{t\leq T} \log(C(t))\sim \log X.$$
So we begin by noting that since $\Delta(t)\sim X$ we have that $$\frac{1}{T}\sum_{t\leq T} \log (\Delta(t))\sim\log X.$$
So we write
\begin{equation*}
\frac{1}{T}\sum_{E_t\in \C}\log C(t)= \frac{1}{T}\sum_{t\leq T}\log \Delta(t)-\frac{1}{T}\sum_{t\leq T}\log\left(\frac{\Delta(t)}{C(t)}\right)
\end{equation*}
and we will show that the second term on the right hand side is in the error term. Let $\nu_p(f(t))$ denote the function such that $p^{\nu_p(f(t))}||f(t)$ then we have that
$$\frac{1}{T}\sum_{t\leq T}\log\left(\frac{\Delta(t)}{C(t)}\right)=\frac{1}{T}\sum_{t\leq T}\sum_{p^{\nu_p(\Delta(t))}||\Delta(t)}\log(p^{\nu_p(\Delta(t))-\nu_p(C(t))}).$$

Since $\Delta(t)=2^4(t^2+3t+9)^2$ we have that $\nu_p(\Delta(t))\geq 2$ for primes $p>2$ and for primes $p>2,3$ we have that $p\mid C(t)$ implies that $p\mid \Delta(t)$. Now suppose $\nu_p(\Delta(t))=2$ then $p || t^2+3t+9$ and thus $p^2|| C(t)$. Hence $$\sum_{t\leq T}\frac{1}{\log X}\sum_{\substack{p\mid \Delta(t)\\\nu_p(\Delta(t))=2}}\log(p^{\nu_p(\Delta(t))-\nu_p(C(t))})=0.$$
Thus,
\begin{align*}
\frac{1}{T}\sum_{t\leq T}\log\left(\frac{\Delta(t)}{C(t)}\right)&=\frac{1}{T}\sum_{t\leq T}\sum_{\substack{p^{\nu_p(\Delta(t))}||\Delta(t) \\ \nu_p(\Delta(t))>2}}\log(p^{\nu_p(\Delta(t))-\nu_p(C(t))})\\
&\leq \frac{1}{T}\sum_{t\leq T}\sum_{\substack{p^{\nu_p(\Delta(t))}||\Delta(t) \\ \nu_p(\Delta(t))>2}}\log(p^{\nu_p(\Delta(t))-2})\\
&\ll\frac{1}{T}\sum_{\substack{p^{\nu_p(\Delta(t))}\ll T^2 \\ \nu_p(\Delta(t))>2}}\log(p^{\nu_p(\Delta(t))-2})\sum_{t\leq \frac{T}{p^{\nu_p(\Delta(t))}}}1\\
&\ll\sum_{\substack{p^{\nu_p(\Delta(t))}\ll T^2 \\ \nu_p(\Delta(t))>2}}\frac{\log(p^{\nu_p(\Delta(t))-2})}{p^{\nu_p(\Delta(t))}}+\frac{1}{T}\sum_{\substack{p^{\nu_p(\Delta(t))}\ll T^2 \\ \nu_p(\Delta(t))>2}}\log(p^{\nu_p(\Delta(t))-2})\\
&\ll T^{\frac{-1}{3}}=\underline{o}(\log X)
\end{align*}
by partial summation.
\end{proof}

Then, using the change of variables \eqref{changeofV},
we rewrite the statement of Theorem \ref{tratiostheoremtwo} as
\begin{eqnarray*}
\frac{1}{|\FX|} \sum_{E \in \FX} \sum_{\gamma_E} \psi \left( \frac{\gamma_E L}{\pi} \right) &\sim&
\int_{-\infty}^\infty \psi (\tau)
 g(\tau) \;d\tau + \phi(0) \\
 &=& \int_{-\infty}^\infty \psi (\tau)
 h(\tau) \;d\tau \\
\end{eqnarray*}
where
\begin{eqnarray*} \nonumber
g(\tau) &=& h(\tau) - \delta_0(\tau) \\ &=&  \frac{1}{2L} \bigg[2 \log\left(\frac{\sqrt{X}}{2\pi}\right) +
\frac{\Gamma'}{\Gamma}\left(1 + \frac{\pi i \tau}{L}\right) +
\frac{\Gamma'}{\Gamma}\left(1 - \frac{\pi i\tau}{L}\right) +2
\bigg\{ -\frac{\zeta'}{\zeta}\left(1 + \frac{2\pi i \tau}{L}\right)
\nonumber \\
&-& \frac{\zeta'}{\zeta}\left(1 + \frac{\pi i \tau}{L}\right)+
A_{\alpha}\left(\frac{\pi i\tau }{L}, \frac{\pi i\tau }{L}\right)
+\exp\left(-\frac{2\pi i\tau}{L}\log\left(\frac{\sqrt{X}}{2\pi}\right)
\right)\nonumber\\
& \times&\frac{\Gamma(1 - \frac{\pi i\tau}{L})}{\Gamma(1 + \frac{\pi
i\tau}{L})} \frac{\zeta(1+\frac{2\pi i\tau}{L})\zeta(1+\frac{\pi
i\tau}{L})}{\zeta(1-\frac{\pi i\tau}{L})} A\left(-\frac{\pi i\tau
}{L}, \frac{\pi i\tau}{L}\right)\bigg\}-\frac{2L}{\pi i\tau}\bigg]. \label{gsecondfamily}
\end{eqnarray*}

We then compute the Taylor expansion of $h(\tau)$ in $L^{-1}$ which gives
\begin{align*}
h(\tau) &= \frac{1}{2L} \Bigg[2L -2\gamma_0+2\bigg[-\left(\frac{-L}{2\pi i \tau}+\frac{-L}{\pi i \tau}+2\gamma_0\right)+ A_{\alpha}(0,0)+O(L^{-1})\\
&+e^{-2\pi i\tau}\left(1+\frac{2\pi \gamma_0 i \tau}{L}+O(L^{-2})\right)\left(\frac{L}{2\pi i\tau}+\gamma_0-\frac{2\pi \gamma_0 i \tau}{L}+O(L^{-2})\right)\\
&\times\left(\frac{L}{\pi i\tau}+\gamma_0-\frac{\pi \gamma_0 i \tau}{L}+O(L^{-2})\right)\left(\frac{-\pi i \tau}{L}+\frac{\pi^2\gamma_0\tau^2}{L^2}+O(L^{-3})\right)\\
&\times\left(1+(A_\gamma(0,0)-A_\alpha(0,0))\frac{\pi i \tau}{L}+O(L^{-2})\right)\bigg]-\frac{2L}{\pi i \tau}\Bigg]+\delta_0(\tau)\\
&=1+\frac{\sin(2\pi \tau)}{2\pi \tau}+\frac{1-\cos(2\pi \tau)}{2\pi i \tau}+\delta_0(\tau)+\frac{A_1(\tau)}{L}+O\left(\frac{1}{L^2}\right)
\end{align*}
where $$A_1(\tau) =
A_\alpha(0,0)-3\gamma_0
+e^{-2 \pi i \tau} \left(\frac{1}{2}\left(A_\alpha(0,0)-A_\gamma(0,0)\right) - 3\gamma_0\right)  .
$$
Since $\psi$ is an even function we have that
$$\int_{-\infty}^\infty \psi(\tau)\left(\frac{1-\cos(2\pi \tau)}{2\pi i \tau}\right)d\tau = 0$$
and hence
\begin{align*}
&\frac{1}{|\FX|} \sum_{E \in \FX} \sum_{\gamma_E} \psi \left( \frac{\gamma_E L}{\pi} \right) \\ \sim&
\int_{-\infty}^\infty \psi(\tau)
\bigg[1 +\delta_0(\tau) +\frac{\sin(2\pi\tau)}{2\pi \tau} + A_1(\tau) L^{-1} + O(L^{-2} ) \bigg] d\tau .
\end{align*}
Then, the leading terms for the one-level scaling density associated to
the  family $\Ft$ given by \eqref{oneparameterfamily} give
\begin{eqnarray*}
\mathcal{W}(\tau) &=& 1 +\delta_0(\tau) +\frac{\sin(2\pi\tau)}{2\pi \tau} \\
&=& \delta_0(\tau) + \mathcal{W}(SO(\mbox{even}))(\tau) .
\end{eqnarray*}
This is a new phenomenon, as previous studies of the one-level density for this
family with other techniques did not allow us to identify the second part of the
density, since the densities associated with $O$, $SO(\mbox{even})$ and $SO(\mbox{odd})$
are undistinguishable for small support of the Fourier transform. This result
may seem surprising at first glance, since we obtain the symmetry type  $SO(\mbox{even})$ for a family of odd rank, but we explain in the next section why this makes sense
via the Birch and Swinnerton-Dyer conjectures.

\section{One-level scaling density for the second family}
\label{weird-density-explained}

We give in this section some explanations for the density
$$\mathcal{W}(\tau) = 1 +\delta_0(\tau) +\frac{\sin(2\pi\tau)}{2\pi \tau}$$
which is associated to the one-parameter family $\mathcal{F}$ of Section \ref{oneparameter}.
There are two pieces for this density, where the first one corresponds to the
contribution of the family zero at the central point. We then write
 $$\mathcal{W}(\tau) = \mathcal{W}_1(\tau)  + \mathcal{W}_2(\tau),$$
where $\mathcal{W}_1(\tau) = \delta_0(\tau).$

We first review the steps that led us to Theorem  \ref{tratiostheoremtwo}.
Using the ratios conjectures, we computed in Section \ref{oneparameter}
the average value of
$$
\frac{ L(1/2+\alpha, E_t)}{ L(1/2+\gamma, E_t)}$$
for the family $\mathcal{F}$.
Writing
$$
L(s, E_t) = \prod_p \left( 1 - \frac{\lambda_t(p)}{p^s} + \frac{\psi(p)}{p^{2s}} \right)^{-1},
$$
we have by Lemma \ref{qstarone}  that the average of $\lambda_t(p)$ over the family is
$$ - \frac{ 1 + \chi_4(p)}{\sqrt{p}}.$$
We then define $\lambda_t^*(p)$ by
$$
\lambda_t(p) = \lambda_t^*(p) - \frac{ 1 + \chi_4(p)}{\sqrt{p}},$$
and the average of $\lambda_t^*(p)$ over the family is $0$ by
Lemma \ref{qstarone}.

Now, for $\RRe{(s)} > 1$,
\begin{eqnarray*}
L(s, E_t) &=& \prod_p \left( 1 - \frac{\lambda_t^*(p)}{p^s} + \frac{\psi(p)}{p^{2s}} +
\frac{1 + \chi_4(p)}{p^{1/2+s}} \right)^{-1} \\
&=&  \prod_p \left( 1 + \frac{\lambda_t^*(p)}{p^s} - \frac{\psi(p)}{p^s} - \frac{1 + \chi_4(p)}{p^{1/2+s}}
+ \frac{\lambda_t^*(p)^2}{p^{2s}} + \mbox{h.o.t.} \right)\\
&=& \prod_p \left( 1 + \frac{\lambda_t^*(p)}{p^s} -  \frac{1 + \chi_4(p)}{p^{1/2+s}}
+ \frac{\lambda_t^*(p)^2 - \psi(p)}{p^{2s}} + \mbox{h.o.t.}\right)
\end{eqnarray*}
and
\begin{eqnarray*}
L(s, E_t)^{-1} &=& \prod_p\left(  1 - \frac{\lambda_t^*(p)}{p^s} + \frac{\psi(p)}{p^{2s}} +
\frac{1 + \chi_4(p)}{p^{1/2+s}} + \mbox{h.o.t.}\right),
\end{eqnarray*}
where the higher order terms are bounded by $p^{-2 \sigma - 1/2 + \varepsilon}$. We will use $\sigma = 1/2$ below, so the higher order terms
do not affect the convergence.

By Lemma \ref{qstarone},
the average over the family of $\lambda_t^*(p)$ is $0$, and by Lemma \ref{qstartwo},
the average
over the family of $\lambda_t^*(p)^2$ is 1.
Then, replacing each expression in the Euler product by its average over the family, we obtained from $\eqref{etavgeuler}$ the ``average Euler product"
\begin{eqnarray*}
\prod_p \left( 1 - \frac{1 + \chi_4(p)}{p^{1+\alpha}} +
\frac{1}{p^{1 + 2 \gamma}} + \frac{1 + \chi_4(p)}{p^{1+\gamma}}  - \frac{1}{p^{1 + \alpha+\gamma}} + \mbox{h.o.t} \right),
\end{eqnarray*}
where the higher order terms give an absolutely convergent product in the neighborhood of $(0,0)$, so the above behaves like
$$
\frac{\zeta(1+\gamma) \zeta(1+2 \gamma)}{\zeta(1+\alpha) \zeta(1 + \alpha + \gamma)},
$$
which is the result of Theorem \ref{factorHdash}. Using the ratios conjecture according to the usual recipe, we got in Section \ref{oneparameter} that the leading terms for
the one-level scaling density were
$$\mathcal{W}(t) = \delta_0(\tau) + 1 + \frac{\sin{(2 \pi \tau)}}{\pi \tau}.$$

Then, in order to isolate the family zero from the previous argument, we
first write
\begin{eqnarray}
\label{firstprod}
L^*(s, E_t) = L(s,E_t) L(s),
\end{eqnarray}
where
$$
L^*(s, E_t) = \prod_p \left( 1 - \frac{\lambda_t^*(p)}{p^s} + \frac{\psi(p)}{p^{2s}} \right)^{-1},$$
and
\begin{eqnarray}  \label{L-function}
L(s) = \frac{L^*(s, E_t)}{L(s,E_t)}
= \prod_p \left( 1 + \frac{1 + \chi_4(p)}{p^{1/2+s}} \right) F(s), \end{eqnarray}
where $F(s)$ converges absolutely for Re$(s) \geq 1/2,$ and has no zeroes in this region.
Furthermore,
$$\prod_p \left( 1 + \frac{1 + \chi_4(p)}{p^{s+1/2}} \right)$$
is related to the Dedekind zeta function
$\zeta_K(s+1/2)$, where $K = \mathbb{Q}(i)$, and we can rewrite (\ref{L-function}) as
\begin{eqnarray} \label{secondprod}
L(s) = \frac{L^*(s, E_t)}{L(s,E_t)} = \zeta_K(s+1/2) F(s), \end{eqnarray}
where $F(s)$ converges absolutely for Re$(s) \geq 1/2,$ and has no zeroes in this region (renaming $F$).
From \eqref{firstprod} and \eqref{secondprod}, the set of zeroes of $L(s, E_t)$ for Re$(s)=1/2$ is the union of the zeroes of $L^*(s, E_t)$ and
the poles of
$\zeta_K(s+1/2)$ for Re$(s)=1/2$.
In other words,
\begin{equation} \label{sum-of-2-densities}
\mathcal{W}(\tau) = \mathcal{W}_1(\tau)  + \mathcal{W}_2(\tau),\end{equation}
where $\mathcal{W}_2(\tau)$ is the density corresponding to the L-functions
$L^*(s, E_t)$
on average for $E_t \in \F$ for the family $\mathcal{F}$ of \eqref{oneparameterfamily},
and
$\mathcal{W}_1(\tau)$ is the density corresponding to the zeroes coming from the poles of
$L(s) = \zeta_K(s+1/2)F(s)$ for Re$(s)=1/2$. Then,  $\mathcal{W}_1(\tau)$ does not
depend of the family, and since
there is only one pole at $s=1/2$,
this gives $$\mathcal{W}_1(\tau) = \delta_0(\tau).$$

We now study the zeroes of $L^*(s, E_t)$. Of course, these are not the L-functions
associated to any elliptic curve, but we can predict the ``rank" of those L-functions
assuming the Birch and Swinnerton-Dyer conjecture for the original L-functions
$L(s, E_t)$. In a nutshell, if the original L-functions have odd rank, then
the L-functions $L^*(s, E_t)$ have even rank, since $\lambda_t(p) = \lambda_t^*(p) - \displaystyle \frac{ 1 + \chi_4(p)}{\sqrt{p}}.$


More precisely, let $E$ be an elliptic curve of rank $r$. With the usual notation,
we have
 $$\lambda_E(p) = \frac{a_E(p)}{\sqrt{p}},$$ and the Birch and Swinnerton-Dyer conjecture \cite{BSD} predicts that
$$\prod_{p \leq x}  \frac{p+1-a_E(p)}{p}  = \prod_{p \leq x}  1 + \frac{1 -  \lambda_E(p) \sqrt{p}}{p}
\sim C (\log{x})^r$$ for some constant $C$ depending on $E$.
Then, for the L-functions $L^*(s, E_t)$, the
Birch and Swinnerton-Dyer conjecture  predicts that
\begin{eqnarray*}
\prod_{p \leq x}  \frac{p+1-a_E^*(p)}{p}
=\prod_{p \leq x}  \frac{p+1-a_E(p) - (1 + \chi_4(p))}{p}
\sim C' (\log{x})^{r-1}
\end{eqnarray*}
where $r$ is the rank of the original curve $E$ and
$C'$ depends on $E$, since
\begin{eqnarray*}
&&\prod_{p \leq x}  \frac{p+1-a_E(p) - (1 + \chi_4(p))}{p} \\
&&= \prod_{p \leq x} \frac{p+1-a_E(p)}{p} \;\; \prod_{p \leq x} \frac{p+1-a_E(p)- (1 + \chi_4(p))}{p+1-a_E(p)} \\
&&= \prod_{p \leq x} \frac{p+1-a_E(p)}{p} \;\;\prod_{p \leq x} \frac{p-1}{p}
\;\;\prod_{p \leq x} \frac{p+1-a_E(p)- (1 + \chi_4(p))}{p+1-a_E(p)} \frac{p}{p-1} \\
&&= \prod_{p \leq x} \frac{p+1-a_E(p)}{p} \;\;\prod_{p \leq x} \frac{p-1}{p}
\;\;\prod_{p \leq x} \frac{p^2 - p a_E(p) - p  \chi_4(p)}{p^2-p a_E(p)+ a_E(p) - 1} \\
&&= \prod_{p \leq x} \frac{p+1-a_E(p)}{p} \;\;\prod_{p \leq x} \frac{p-1}{p} \;\;
\;\;\prod_{p \leq x} 1 + \frac{- p  \chi_4(p)-a_E(p)+1}{p^2-p a_E(p)+ a_E(p) - 1} \\
&&\sim C' (\log{x})^{r-1}.
\end{eqnarray*}

Then, since $r$ was odd for the original family, the L-functions $L^*(s, E_t)$
behave like a family of {\it even} rank, and we should have
$$\mathcal{W}_2(\tau) = 1 + \frac{\sin{(2 \pi \tau)}}{2 \pi \tau},$$
in \eqref{sum-of-2-densities}, and
\begin{eqnarray*}
\mathcal{W}(\tau) &=& \mathcal{W}_1(\tau) + \mathcal{W}_2(\tau) = \delta_0(\tau) + 1 + \frac{\sin{(2 \pi \tau)}}{2 \pi \tau} \\ &=&
\delta_0(\tau) +  \mathcal{W}(SO(\mbox{even}))(\tau).\end{eqnarray*}

\newpage
\kommentar { In \reff{oneleveldensity} we set
\begin{align} \nonumber
W(t) & := \frac{1}{|\FX|}\frac{1}{2\pi} \sum_{E\in \FX}\bigg[ 2
\log\bigg(\frac{\sqrt{N_E}}{2\pi} \bigg) +
\frac{\Gamma'}{\Gamma}(1+it) \\ \nonumber
& + \frac{\Gamma'}{\Gamma}(1-it) + 2 \bigg(-\frac{\zeta'}{\zeta}(1+2it) + A_{\alpha}(it,it) \\
&
-\omega_E\left(\frac{\sqrt{N_E}}{2\pi}\right)^{-2it}\frac{\Gamma(1-it)}{\Gamma(1+it)}
\zeta(1+2it)A(-it,it)\bigg) \bigg].
\end{align}
Thus, we have
\begin{equation}
D(\phi, \FX) = \int_{-\infty}^\infty \phi(t) W(t) dt +
O(X^{-1/2+\varepsilon}).
\end{equation}
Young showed that $\FX$ has orthogonal symmetry. }


\end{document}